\documentclass{amsart}
\usepackage[scale=0.75, centering, headheight=14pt]{geometry}
\usepackage[latin1]{inputenc}
\usepackage[T1]{fontenc}
\usepackage{lmodern}
\usepackage[english]{babel}
\usepackage{microtype}
\usepackage{amsmath,amssymb,amsfonts,amsthm}
\usepackage{mathtools,accents}
\usepackage{mathrsfs}
\usepackage{aliascnt}
\usepackage{braket}
\usepackage{bm}
\usepackage[initials]{amsrefs}

\usepackage[citecolor=blue,colorlinks]{hyperref}
\addto\extrasenglish{}

\usepackage{enumerate}
\usepackage{xcolor}

\usepackage{aliascnt}

\makeatletter
\def\newaliasedtheorem#1[#2]#3{
  \newaliascnt{#1@alt}{#2}
  \newtheorem{#1}[#1@alt]{#3}
  \expandafter\newcommand\csname #1@altname\endcsname{#3}
}
\makeatother

\theoremstyle{plain}
\newtheorem{theorem}{Theorem}[section]
\newaliasedtheorem{lemma}[theorem]{Lemma}
\newaliasedtheorem{prop}[theorem]{Proposition}
\newaliasedtheorem{claim}[theorem]{Claim}
\newaliasedtheorem{corollary}[theorem]{Corollary}

\theoremstyle{definition}
\newaliasedtheorem{definition}[theorem]{Definition}
\newaliasedtheorem{example}[theorem]{Example}

\theoremstyle{remark}
\newaliasedtheorem{remark}[theorem]{Remark}

\numberwithin{equation}{section}

\def\R{\mathbb R}
\def\N{{\mathbb N}}% nonnegative integers
\def\ra{\rightarrow}

\title[Scattering for Dipolar BEC]{Asymptotic dynamic for Dipolar Quantum Gases below the ground state energy threshold}

\author[J. Bellazzini \and L. Forcella]{Jacopo Bellazzini \and Luigi Forcella}

\address{Jacopo Bellazzini \hfill\break  Universit\`a di Sassari, via Piandanna 4, 07100 Sassari, Italy}
\email{jbellazzini@uniss.it}
\address{Luigi Forcella\hfill\break  \'Ecole Polytechnique F\'ed\'erale de Lausanne, Institute of Mathematics, Station 8, CH-1015 Lausanne, Switzerland.}
\email{luigi.forcella@epfl.ch}

\subjclass[2000]{35Q55, 35B40, 82C10, 35J20}

\keywords{Gross-Pitaevskii equation, BEC, NLS-like equation, asymptotic behaviour, concentration/compactness method}

\begin{document}

\maketitle

%\tableofcontents{}

\begin{abstract}
We consider the Gross-Pitaevskii equation describing a dipolar Bose-Einstein condensate without external confinement. We first consider the unstable regime, where the nonlocal nonlinearity is neither positive nor radially symmetric and standing states are known to exist. We prove that under the energy threshold given by the ground state, all global in time solutions behave as free waves asymptotically in time. The ingredients of the proof are variational characterization of the ground states energy, a suitable profile decomposition theorem and localized virial estimates, enabling to carry out a Concentration/Compactness and Rigidity scheme. As a byproduct we show that in the stable regime, where standing states do not exist, any initial data in the energy space scatters.
\end{abstract}

\section{Introduction}
The interest concerning the asymptotic dynamics of equations describing  a condensate of particles at very low temperatures speedily increased since the first experimental observation in $1995$  of Bose-Einstein condensate (BEC), see e.g. \cite{AEMWC,BrSaToHu}.
In the recent years, the so-called dipolar Bose-Einstein condensate, namely a condensate made out of particles possessing a permanent electric or magnetic dipole moment, see e.g. \cite{ BaCa,BaCaWa, NaPeSa,  SSZL}, has been attracting much attention. At temperatures much smaller than the critical one, such a model is well described by the wave function $u=u(t,x)$ whose evolution is governed by the Gross-Pitaevskii equation (GPE), 
\begin{equation}
\label{eq:evolution}
i h \frac{\partial u}{\partial t} = - \frac{h^2}{2m}\nabla^2 u + W(x) u + U_0|u|^2 u + (V_{dip}\ast |u|^2) u, \quad x \in \R^3, \quad t>0,
\end{equation}
where $t$ is the time variable, $x = (x_1,x_2,x_3)$ is the space coordinate, $h$ is the Planck constant, $m$ is the mass of a dipolar particle and $W(x)$ is an external, trapping, real potential. In this paper we consider the case when the trapping potential $W$ is not active, i.e. assuming $W(x)=0$. The coefficient $U_0 = 4 \pi h^2 a_s /m$ describes the local interaction between dipoles in the condensate with $a_s$ the $s$-wave scattering length (positive for repulsive interactions and negative for attractive interactions).

The long-range dipolar interaction potential between two dipoles is given by
\begin{equation*}
V_{dip}(x) = \frac{\mu_0 \mu^2_{dip}}{4 \pi} \,  \frac{1 - 3\cos^2 (\theta)}{|x|^3}, \quad x \in \R^3,
\end{equation*}
where $\mu_0$ is the vacuum magnetic permeability, $\mu_{dip}$ is the permanent magnetic dipole moment and $\theta$ is the angle between the dipole axis $n$ and the vector $x$.  For simplicity, we fix the dipole axis as the vector $n=(0,0,1)$.
The wave function is normalized according to
\begin{equation*}
\int_{\R^3} |u(x,t)|^2\,dx = N,
\end{equation*}
where $N$ is the total number of dipolar particles in the dipolar BEC.  \\

In order to simplify the mathematical analysis 
we rescale \eqref{eq:evolution} into the following dimensionless GPE,
\begin{equation}\label{GP}
\left\{ \begin{aligned}
i\partial_{t}u+\frac12\Delta u&=\lambda_1|u|^{2}u+\lambda_2(K\ast|u|^2)u, \quad (t,x)\in \R\times 
\mathbb{R}^3\\
u(0,x)&=u_0(x)\in H^1(\mathbb{R}^3)
\end{aligned}\right..
\end{equation}
The corresponding normalization  now reads
\begin{equation*}
N(u(\cdot, t)):= \|u(\cdot, t)\|^2_2 = \int_{\R^3}|u(x,t)|^2\,dx = \int_{\R^3}|u(x,0)|^2\,dx = 1,
\end{equation*}
and the kernel $K$ is given by 
\begin{equation*}
K(x)=\frac{x_1^2+x_2^2-2x_3^2}{|x|^5}.
\end{equation*}
The physical real parameters $\lambda_{1,2}$, which describe the strength of the two nonlinearities, are given by
\begin{equation*}
\lambda_1 = 4 \pi a_s N \gamma, \quad \lambda_2 = \frac{mN \mu_0 \mu_{dip}^2 }{4 \pi h^2}\gamma.
\end{equation*}

In this paper we consider the case when the two real parameters $\lambda_{1,2}$  range in the so called \emph{unstable regime}:
\begin{equation}\label{UR}
\left\{ \begin{aligned}
\lambda_1-\frac{4\pi}{3}\lambda_2<0 &\quad \hbox{ if } \quad \lambda_2>0\\
\lambda_1+\frac{8\pi}{3}\lambda_2<0 & \quad\hbox{ if } \quad\lambda_2<0
\end{aligned}\right..
\end{equation}
Solutions $u(t)\in\mathcal C((-T_{min},T_{max});H^1(\mathbb R^3))$ to \eqref{GP} have been proved to exist, at least locally in time, by Carles, Markovich and Sparber in  \cite{CMS}, and not only in the unstable regime but also in the complement region 
\begin{equation}\label{SReg}
\left\{ \begin{aligned}
\lambda_1-\frac{4\pi}{3}\lambda_2\geq 0 &\quad \hbox{ if } \quad \lambda_2>0\\
\lambda_1+\frac{8\pi}{3}\lambda_2\geq 0 & \quad\hbox{ if } \quad\lambda_2<0
\end{aligned}\right.,
\end{equation}
\noindent which is called the \emph{stable regime}.\\

%\We now recall that since solutions to \eqref{GP} conserve mass and energy.
We now recall that solutions to \eqref{GP} conserve mass and energy, namely 
\begin{equation}\label{eq:mass}
\mathcal M(t)=\mathcal M(u(t)):=\int_{\R^3}|u(t)|^2\,dx=\mathcal M(0)
\end{equation}
and
\begin{equation}\label{energy}
\mathcal E(t)=\mathcal E(u(t)):=\frac{1}{2}\left(\int_{\R^3}|\nabla u(t)|^2+\lambda_1|u(t)|^4+\lambda_2(K\ast|u(t)|^2)|u(t)|^2\,dx\right)=\mathcal E(0)
\end{equation} 
for any $t\in(-T_{min},T_{max}),$ where $T_{min},T_{max}\in(0,\infty]$ are the minimal and maximal time of existence, respectively.  \\

The unstable regime is of particular interest since stationary solutions are allowed in this region.  More precisely, stationary states  are solutions of the type
\begin{equation*}
u(x,t) = e^{-i \kappa  t}u(x), 
\end{equation*}
where $\kappa \in \R$ is the chemical potential, $u(x)$ is a time-independent function solving the stationary equation
\begin{equation}
\label{eq:maina}
- \frac{1}{2}\Delta u + \lambda_1 |u|^2 u + \lambda_2 (K \ast  |u|^2) u + \kappa u =0
\end{equation}
constrained on the manifold $S(1),$ where  
\begin{equation}\label{constraint1}
S(1) = \{ u \in H^1(\R^3) \ s.t. \ \|u\|_{L^2(\R^3)}^2 =1\}.
\end{equation}
There are two different approaches to show the existence of standing states.\\
The first one is due to Antonelli and Sparber, see \cite{AS}, where existence is proved  by means of the  Weinstein method, i.e. as
minimizers of the following 
scaling invariant functional
\begin{equation*}
J(v):=\frac{\|\nabla v\|^3_{L^2(\R^3)}\|v\|_{L^2(\R^3)}}{-\lambda_1\|v\|_{L^4(\R^3)}^4-\lambda_2\int_{\R^3} (K\ast |v|^2)|v|^2\,dx}.
\end{equation*}
The second strategy, due to the first author and Jeanjean, see \cite{BJ},  relies on  topological methods, more precisely  by proving  the existence
of critical points of the energy functional under the mass constraint depicted in \eqref{constraint1}. In this approach the parameter $\kappa$ is found as Lagrange multiplier. 
Despite the fact the energy is unbounded from below on $S(1)$, if one  restricts to states    that are stationary for the evolution equation, i.e. fulfilling  \eqref{eq:maina},  then the energy is bounded from below by a  positive constant; furthermore, this constant, corresponding to the mountain pass level, is reached. The mountain pass solutions hence correspond  to least energy states, also called ground states. As a direct consequence of this variational characterization and using a virial approach, the associated standing waves are proved to be orbitally unstable.

In \cite{BJ} is also proved that for sufficiently small initial data  in the $H^1(\mathbb R^3)$-norm, then (global) solutions to \eqref{GP} scatter (for the formal definition, see \autoref{def:scat}), no matter if the equation is considered in the unstable regime \eqref{UR} or not. 

Our aim is to study the long time behaviour of global solutions to \eqref{GP} subject to condition \eqref{UR} up to some threshold given in term of the ground state energy, by removing therefore the assumption on the smallness of the initial data. Our strategy follows the 
Kenig and Merle scheme,  developed in the well celebrated papers \cite{KM1, KM2} to solve the global existence and scattering problems for the energy critical, focusing, radial nonlinear Schr\"odinger and Wave Equations in low spatial dimensions, respectively.

Before stating our main result we recall the rigorous definition of scattering. In the following $U(t)f=e^{it\frac{\Delta}{2}}f$ will denote, with standard notation, the linear evolution driven by the free Schr\"odinger propagator of an initial datum $f,$ namely $L(t,x):=U(t)f$ satisfies $i\partial_tL+\frac12\Delta L=0,$ $L(0)=f.$ 
\begin{definition}\label{def:scat}
Let $u_0\in H^1(\mathbb{R}^3)$ be given and  $u(t, x)\in \mathcal {C}(\mathbb{R};H^1(\mathbb{R}^3))$
be the corresponding  unique global solution (if it exists) to \eqref{GP}. Then we say that 
$u(t,x)$  \emph{scatters} provided 
\begin{equation*}
\lim_{t\rightarrow \pm \infty} \|u(t,x) - e^{it\frac{\Delta}{2}} u^\pm \|_{H^1 (\mathbb{R}^3)}=0,
\end{equation*}
for suitable $u^\pm\in H^1(\mathbb{R}^3)$.
\end{definition}
We point out that is not guaranteed that solutions to \eqref{GP} do exist globally in time, so in the analysis below we shall also give sufficient conditions such that local in time solutions to \eqref{GP} (whose existence has been shown in \cite{CMS}) can be extended globally in time. 
Let us recall some notation introduced in \cite{BJ}: the quantity defined in \eqref{energy} can be rewritten as
\begin{equation*}
\mathcal E(t)=\frac{1}{2}\int_{\R^3}|\nabla u|^2\,dx+\frac{1}{2(2\pi)^3}\int_{\R^3}\left(\lambda_1+\lambda_2\hat K(\xi)\right)(\widehat{|u|^2} )^2(\xi)\,d\xi
\end{equation*}
by means of Plancherel identity, where  the Fourier transform of $K$ is given by
\begin{equation}\label{kernel:fou}
\hat K(\xi)=\frac{4\pi}{3}\frac{2\xi_3^2-\xi_2^2-\xi_1^2}{|\xi|^2}
\end{equation}
(see \cite{CMS} for a proof of the explicit form of $\hat K.$) A trivial computation  leads to 
\begin{equation*}
\hat K\in\left[-\frac43\pi,\frac83\pi\right].
\end{equation*} 
We split the energy as sum of the following kinetic and potential energies, respectively defined by 
\begin{align}\label{kinetic:en}
\mathcal T(u)&=\int_{\R^3}|\nabla u|^2\,dx\\ \label{potential:en}
\mathcal P(u)&=\frac{1}{(2\pi)^3}\int_{\R^3}\left(\lambda_1+\lambda_2\hat K(\xi)\right)(\widehat{|u|^2})^2(\xi)\,d\xi,
\end{align}
\noindent and we introduce the quantity (suggested by the Pohozaev identities)
\begin{equation}\label{GG}
\mathcal  G(u)=\mathcal T(u)+\frac32\mathcal P(u).
\end{equation}

Despite the fact that we are primarily interested in solutions satisfying \eqref{constraint1}, for the mathematical treatment of the problem it is convenient to consider the generic set of constraints
\begin{equation*}
S(c)=\left\{ u \in H^1(\R^3) \ s.t. \ \|u\|_{L^2(\R^3)}^2=c\right\}.
\end{equation*}
Here $c>0$ and the case $c=1$ trivially corresponds to the normalization \eqref{constraint1}. Given $c>0$, $\mathcal{E}(u)$ has a mountain pass geometry on $S(c)$ (see the monograph \cite{AM} for a detailed treatment of this topic). More precisely, there exists  $\beta>0$ such that 
\begin{equation*}
\gamma(c) := \inf_{g \in \Gamma(c)} \max_{t\in [0,1]}\mathcal{E}(g(t)) > \max \left\{\max_{g \in \Gamma(c)}\mathcal{E}(g(0)), \max_{g \in \Gamma(c)}\mathcal{E}(g(1))\right\}
\end{equation*}
holds in the set
\begin{equation*}
\Gamma(c) =\left\{ g \in C([0,1]; S(c))  \ s.t. \ g(0) \in A_{\beta},\mathcal{E}(g(1))<0 \right\},
\end{equation*}
where
\begin{equation*}
A_{\beta}= \left\{u \in S(c) \ s.t. \ \left \| \nabla u \right \|_{L^2(\R^3)}^2\leq \beta\right\}.
\end{equation*}

It it standard, see \cite{AM}, that the mountain pass geometry induces the existence of a Palais-Smale sequence at the level $\gamma(c),$ namely a sequence $\{u_n\}_{n\in\N} \subset S(c)$ such that, as $n\to\infty,$
\begin{equation*}
\mathcal{E}(u_n)=\gamma(c)+o(1),\qquad \|\mathcal{E}'|_{S(c)}(u_n)\|_{H^{-1}}=o(1).
\end{equation*}
If one can show in addition the compactness of $\{u_n\}_{n\in\N}$, namely that up to a subsequence,  $u_n \rightarrow u$ in $H^1(\R^3)$, then a critical point is found at the level $\gamma(c)$. This is exactly what happens under the assumptions \eqref{UR}.  We can summarize the last paragraph in the following.

\begin{theorem}\label{thm:standing}\cite[Theorem 1.1]{BJ}
Let $c>0$ and assume that \eqref{UR} holds. Then $\mathcal{E}(u)$ has a {\it mountain pass geometry} on $S(c)$ and there exists a couple $\{u_c, \kappa_c\}\in H^1(\R^3) \times \R^{+}$ solution of \eqref{eq:maina} with 
$\|u_c\|_{L^2(\R^3)}^2=c$ and $E(u_c)=\gamma(c).$ In addition $u_c \in S(c)$ is a ground state. 
\end{theorem}
We moreover recall that the energy level $\gamma(c)$  has the following variational characterization that will be crucial to establish the scattering result below:
\begin{equation}\label{def:vc}
\gamma(c)=\inf\{\mathcal E(u) \ s.t. \ u\in V(u)\} \hbox{\quad where \quad} V(c)=\{u\in H^1(\R^3) \ s.t. \ \|u\|_{L^2(\R^3)}^2=c  \hbox{ and } \mathcal G(u)=0\}.
\end{equation} 
In \cite{BJ}, among the other results, the following sufficient conditions are given in order to have global existence of solution to \eqref{GP}.
\begin{theorem}\label{thm:1}{\cite[Theorem 1.3]{BJ}}
If $\mathcal E(u_0)<\gamma(c),$ with $c=\|u_0\|_{L^2(\R^3)}^2$ and $\mathcal G(u_0)>0,$ then the (local in time) solution $u\in\mathcal C((-T_{min},T_{max});H^1(\R^3))$ to \eqref{GP} can be extended globally in time, i.e. $T_{min}=T_{max}=\infty,$ and 
\begin{equation*}\mathcal G(u(t))>0
\end{equation*} 
for all $t\in\R$.
\end{theorem}

Our aim in this paper is to show that something more can actually be said about the solution to \eqref{GP} under the conditions of \autoref{thm:1}. In fact, in that region all solutions scatter. The main theorem of the paper is then as follows.

\begin{theorem}\label{main} 
For any initial datum $u_0\in H^1(\mathbb R^3)$ satisfying $\mathcal E(u_0)<\gamma(c),$ with $c=\|u_0\|_{L^2(\R^3)}^2$ and $\mathcal G(u_0)>0,$ then the corresponding global solution to \eqref{GP} scatters.
\end{theorem}
\begin{remark}
It is worth mentioning that we do not assume neither finite variance nor spherical symmetry of the solutions. Indeed in the radial setting, the equation would reduce to a classical cubic NLS due to the fact that the nonlocal nonlinearity can be defined as a Calder\'on-Zigmund operator with kernel $|x|^{-3}\mathcal O(x)$ where $\mathcal O$ is a  zero-order function having zero average on the sphere (see \cite{CMS}).
\end{remark}
\begin{remark}
It shall be emphasised that the fact that $\lambda_1, \lambda_2$ are in the unstable regime \eqref{UR} does not imply that the potential energy $\mathcal{P}$ defined in \eqref{potential:en} is negative for any function in $H^1(\R^3)$, see \autoref{eq:contr} below. Despite the fact that when the potential energy $\mathcal{P}$ is positive the nonlinear term acts as a defocusing nonlinearity, we are not able to exclude 
that along the time evolution  $\mathcal{P}(u(x,t))$ changes sign. For this reason the conditions $\mathcal{G}(u_0)>0$ and $\mathcal{E}(u_0)<\gamma(c)$ are necessary even when $\mathcal{P}(u_0)>0$.
\end{remark}

The proof of \autoref{main} is based on the concentration/compactness and rigidity argument. \\
We recall briefly the general strategy (based on a contradiction argument) of the Kenig and Merle road map. As already recalled, in \cite{BJ} it has been proved that for sufficiently small $H^1(\R^3)$ initial data, solutions are global and scatter. Suppose now that the threshold for scattering is strictly below the claimed one. The tool called profile decomposition, based on concentration/compactness principles, proves the existence of a global but non-scattering solution (the so called \emph{minimal element} or \emph{soliton-like solution}, that we denote $u_{sl}$) at the threshold between scattering and non-scattering. Secondly, it is proved that  the flow of this minimal element is (up to some symmetries) a precompact subset of $H^1(\R^3)$ and that therefore it 
remains spatially localized uniformly in time along a continuous path $x(t)\in\R^3$. This uniform localization enables the use of a local virial identity to establish a strictly positive lower bound on the convexity (in
time) of the localized variance. More precisely,  once defines the localized variance  as $z_R(t)=R^2\int_{\R^3}\chi\left(\frac xR\right)|u_{sl}(t,x)|^2\,dx,$ where $\chi\in\mathcal C_c^\infty(\R^3)$ is a suitable cut-off function and $R>1$ is a  rescaling parameter. The goal is to connect the second derivative of this quantity with the function $\mathcal{G},$ introduced in \eqref{GG}, as follows
\begin{equation}\label{lb:GG1}
\frac{d^2}{dt^2}z_R(t)=4 \mathcal{G}(u_{sl}(t))+o(1),
\end{equation}
where the error term decays, uniformly in time, as $R$ increases. A lower positive bound on $\mathcal{G}(u_{sl}(t))$, following from the variational characterization of the mountain pass energy and from the properties of the minimal element, finally permits one to exhibit  a contradiction, since such a minimal element is forced to be the null function, while by construction it is non trivial.
Moreover we underline that for any initial datum $u_0\in H^1(\mathbb R^3)$ satisfying $\mathcal E(u_0)<\gamma(c),$ with $c=\|u_0\|_{L^2}^2$ and $\mathcal G(u_0)>0,$  using our variational approach we are able to derive an explicit lower bound on $\mathcal{G}(u(t))$ for  the corresponding global solution
given by 
\begin{equation}\label{lb:GG}
\mathcal{G}(u(t))\geq \min\{\gamma(c)-\mathcal{E}(u_0), \mathcal{E}(u_0)\}.
\end{equation}
\begin{remark}
From the identity $\mathcal{E}-\frac{1}{3}\mathcal{G}=\frac 16 \mathcal{T}$, it is evident that $\mathcal{E}>0$ if $\mathcal{G}>0.$ Our lower bound on $\mathcal{G}(u(t))$ follows 
only from the variational characterization of the mountain pass energy and not from the fact that this critical energy level is achieved by the ground state. Despite the fact that in  \autoref{prop:coinc} we show that our conditions for scattering coincide with the one in the Kenig and Merle approach, we never use the fact that the ground state exists. 
\end{remark}
The main difficulty concerning  \eqref{lb:GG1} is clearly the presence of the nonlocal dipolar term that makes the analysis more delicate with respect to local nonlinearities. In particular, for the dipolar interaction term, despite the nice identity
 $x\cdot \nabla K(x)=-3 K(x)$,  one cannot use the
brutal estimate $\left|x\cdot \nabla K(x)\right|=3|K(x)|$ due to the singularity of the
kernel at the origin. This argument is, as a matter of fact, the one that simplifies the computations for nonlocal
nonlinearities with positive kernel like for the Coulomb kernel of the form $K(x)=\frac{1}{|x|}$. In our case this rough estimate is not allowed.\\

We conclude this introduction by pointing out how in the stable regime \eqref{SReg}
the potential energy $\mathcal{P}$ is nonnegative  for any function in $H^1(\R^3)$. Indeed let us the consider $\lambda_2>0$, $\lambda_1-\frac{4\pi}{3}\pi\lambda_2>0,$ then we have 
\begin{equation*}
\mathcal{P}(u)=\frac{1}{(2\pi)^3}\int_{\R^3}\left(\lambda_1+\lambda_2\hat K(\xi)\right)(\widehat{|u|^2})^2(\xi)\,d\xi\geq \frac{1}{(2\pi)^3}\int_{\R^3}\left(\lambda_1-\frac{4\pi }{3}\lambda_2\right)(\widehat{|u|^2})^2(\xi)\,d\xi\geq 0.
\end{equation*}
From this fact it is clear that stationary states are not allowed in this regime due to the fact that $V(c)=\emptyset.$
Hence following verbatim the argument for the unstable regime without any other assumptions, we have  the following result (see  \cite[Section 7]{DHR} for an analogous remark about the cubic defocusing NLS).
\begin{corollary}
In the stable regime for any initial datum in $H^1(\R^3)$, the corresponding global solution to \eqref{GP} scatters.
\end{corollary}
\subsection{Notations}\label{notations}
In what follows, we will use the notations below. \\

\noindent 
For $1\leq p\leq \infty,$ the $L^p=L^p(\mathbb R^n;\mathbb C)$ are the classical Lebesgue spaces, %endowed with norm $\|f\|_{L^p}=\left(\int_{\mathbb R^n}|f(x)|^p\,dx\right)^{1/p}$ if $p\neq\infty$ or $\|f\|_{L^\infty}=\esssup_{x\in\mathbb R^n}|f(x)|$ for $p=\infty.$
while $W^{1,r}=W^{1,r}(\mathbb R^3;\mathbb C)$ is defined as the space of function in $L^p$ with distributional derivatives in $L^p,$ with the usual norm $\|f\|_{W^{1,p}}^p=\|f\|^p_{L^p}+\|\nabla f\|^p_{L^p}.$  When $r=2,$ we set $W^{1,2}:=H^1=H^1(\mathbb R^3;\mathbb C).$%:=(1-\Delta)^{-1/2}L^2$ and its homogeneous version $\dot H^1=\dot H^1(\mathbb R^3;\mathbb C):=(-\Delta)^{-1/2}L^2.$ Since we work on $\R^3,$ we simply denote $\int=\int_{\R^3.}$
\\

\noindent Given an interval $I\subseteq \mathbb R,$ bounded or unbounded, we define by $L^p_IX=L^p(I;X)$ the Bochner space of vector-valued functions $f:I\mapsto X$ endowed with the norm $\|f\|_{L^p_IX}=\|f\|_{L^p_t(I;X)}=\left(\int_{I}\|f(s)\|_{X}^p\,ds\right)^{1/p}$ for $1\leq p<\infty,$ with similar modification as above for $p=\infty.$ In case $I=\R$ we simply write $L^pX.$  %(In what follows, $f\in L^pX$ means that $f=f(t,x)$ is a function depending on the time variable $t\in\R$ and the space variable $x\in\R^3,$ and $\|f\|_{L^p_IX}=\|f\|_{L^p_t(I;X_x)}$).
\\

\noindent For a normed (Banach) space $(X,\|\cdot\|)$ we denote by $B_R(X)$ the open ball of radius $R$ with center at the origin, i.e. $B_R(X)= \{f \in X; \|f\| < R \};$ if $X=\R^3,$ then $B(x_0,R)$ is the ball of radius $R$ centered at $x_0.$  
\\

\noindent The operator $\mathcal Ff(\xi)=\hat f(\xi)$ is the standard Fourier Transform, $\mathcal F^{-1}$ being its inverse.
The operator $\tau_y$  denotes the translation operator $\tau_yf(x):=f(x-y),$ while $f\ast g$ is the convolution operator between $f$ and $g.$ $\Re{z}$ and $\Im{z}$ are the common notations for the real and imaginary parts of a complex number $z$. 
\\

\noindent Given a measurable set $\mathcal O\subseteq\R^d,$ $1_{\mathcal O}(x)$ is the indicator function of $\mathcal O.$\\

\noindent Given two quantities $A$ and $B,$ we denote $A \lesssim B$ ($A\gtrsim B,$ respectively) if there exists a positive constant $C$ such that $A \leq CB$ ($A\geq CB,$ respectively). If both the relations hold true, we write $A\sim B.$
\\

\noindent Finally,  for $1\leq p\leq \infty,$ we denote  its conjugate by $p^\prime:=\frac{p}{p-1},$ and since we work on $\R^3,$ we simply denote $\int=\int_{\R^3}.$

\section{Variational Estimates}\label{sec:var:est}

We shall notice that the fact that $\lambda_1, \lambda_2$ belong to the unstable regime does not guarantee that the potential energy
$\mathcal{P}(u)$ fulfils the condition $\mathcal{P}(u)<0$. As an example of this fact we show the following.
\begin{lemma}\label{eq:contr}
Let $\lambda_1, \lambda_2,$ belong to the unstable regime \eqref{UR} and without loss of generality  $\lambda_2>0.$ Assume moreover the additional restriction  $\lambda_1+\frac{8}{3}\pi \lambda_2>0.$ Then there exists $u \in H^1$  with $\|u\|_{L^2}^2=c$ such that 
\begin{equation*} 
\mathcal{P}(u)>0.
\end{equation*}
\end{lemma}
\begin{proof}
We will show the existence of a function $u \in H^1$  with $\|u\|_{L^2}^2=c$ such that $\mathcal{P}(u)>0$ by scaling argument. Let us consider
\begin{equation*}
u^\mu=\mu u(\mu^{1/2} x_1, \mu^{1/2} x_2, \mu x_3).
\end{equation*}
This transformation preserves the $L^2$-norm and is such that the potential energy rescales straightforwardly as
\begin{equation*}
\mathcal{P}(u^\mu)=c \mu^2\int\left(\lambda_1+\frac43 \pi \lambda_2 \frac{2\mu^2 \xi_3^2-\mu\xi_1^2-\mu\xi_2^2}{\mu\xi_1^2+\mu\xi_2^2+\mu^2 \xi_3^2}\right)(\widehat{|u|^2})^2\,d\xi,
\end{equation*}
where we used the scaling property of the Fourier transform for rescaled functions.
Now, with our assumptions, we have 
\begin{equation*}
\lim_{\mu \rightarrow \infty }\lambda_1+\frac43 \pi \lambda_2 \frac{2\mu^2 \xi_3^2-\mu\xi_1^2-\mu\xi_2^2}{\mu\xi_1^2+\mu\xi_2^2+\mu^2 \xi_3^2}=\lambda_1+\frac{8}{3}\pi \lambda_2>0,
\end{equation*}
which implies that 
$\lim_{\mu \rightarrow \infty}\mathcal{P}(u^\mu)=+\infty$ for all $u$ thanks to the Lebesgue theorem.
\end{proof}
It is important to notice that if the potential energy is negative it is possible to introduce a Weinstein-like functional analogous to the one arising from the Gagliardo-Nirenberg inequality,
whose maximizers correspond to the ground states.
In \cite{AS} is proved the following result ensuring existence of a minimizer for the functional below:
\begin{equation}\label{wein:fun}
J(v)=\frac{\|\nabla v\|_{L^2}^3\|v\|_{L^2}}{-\lambda_1\|v\|_{L^4}^4-\lambda_2\int (K\ast |v|^2)|v|^2\,dx}.
\end{equation}
More precisely the result is as follows.
\begin{prop}{\cite[Proposition 3.2]{AS}}
Under the hypothesis \eqref{UR} there exists a minimizer $v_m\in H^1$ to \eqref{wein:fun}, namely 
$\inf \{J(v) \ s.t.\ v\in H^1, v\neq 0 \hbox{ and } \mathcal{P}(v)<0\}$ is attained at $v_m,$ i.e. $J(v_m)=m:=\inf \{J(v)\ s.t.\   v\in H^1, v\neq 0 \hbox{ and } \mathcal{P}(v)<0 \}.$ Moreover, by the scaling invariance property of the functional $J,$ it can be assumed that $\|v_m\|_{L^2}=\|\nabla v_m\|_{L^2}=1.$ Furthermore  $v_m$ solves 
\begin{equation}\label{eq:min}
-3\Delta v_m+4m\left(\lambda_1|v_m|^2v_m+\lambda_2(K\ast|v_m|^2)v_m\right)+v_m=0.
\end{equation}
\end{prop}

%\noindent By plugging $Q(x)=2\sqrt m v_m(6^{-1/2}x)$ into \eqref{eq:min} one finds that such $Q$ satisfies
%\begin{equation*}
%-\frac12\Delta Q+\left(\lambda_1|Q|^2Q+\lambda_2(K\ast|Q|^2)Q\right)+Q=0.
%\end{equation*}

Let us observe that if we have a solution $V$ to $-\alpha\Delta V+\beta\left(\lambda_1|V|^2V+\lambda_2(K\ast|V|^2)V\right)+V=0$ then $W=\gamma V(\rho x)$ satisfies 
\begin{equation*}
\begin{aligned}
-\frac12\Delta W&=-\frac{\gamma\rho^2}{2}\frac\alpha\alpha\Delta V(\rho x)=\frac{\gamma\rho^2}{2\alpha}\left(-\beta\left(\lambda_1|V|^2V+\lambda_2(K\ast|V|^2)V\right)(\rho x)-V(\rho x)\right)\\
&=\frac{\gamma\rho^2}{2\alpha}\left(-\frac{\beta}{\gamma^3}\left(\lambda_1|W|^2W+\lambda_2(K\ast|W|^2)W\right)-\frac1\gamma W\right)\\
&=-\frac{\rho^2\beta}{2\alpha\gamma^2}\left(\lambda_1|W|^2W+\lambda_2(K\ast|W|^2)W\right)-\frac{\rho^2}{2\alpha} W
\end{aligned}
\end{equation*}
and so by choosing 
\[
\rho^2=2\alpha \implies \frac{\rho^2\beta}{2\alpha\gamma^2}=\frac{\beta}{\gamma^2}
\]
and
\[
\gamma=\sqrt\beta
\]
we get 
\begin{equation*}
\begin{aligned}
-\frac12\Delta W+\left(\lambda_1|W|^2W+\lambda_2(K\ast|W|^2)W\right)+W=0.
\end{aligned}
\end{equation*}
Hence, with $V=v_m,$  $\alpha=3$ and $\beta=4m$ we get that $Q(x)=2\sqrt{m}v_m(\sqrt6 x)$ satisfies 
\begin{equation}\label{eq:Q}
-\frac12\Delta Q+\left(\lambda_1|Q|^2Q+\lambda_2(K\ast|Q|^2)Q\right)+Q=0.
\end{equation}
Consider therefore $Q$ be the minimizer for the Weinstein functional \eqref{wein:fun} which fulfils \eqref{eq:Q}. We have the following.

%Let $Q$ be the minimizer for the Weinstein functional \eqref{wein:fun} which fulfils
%\begin{equation*}
%-\frac12\Delta Q+\left(\lambda_1|Q|^2Q+\lambda_2(K\ast|Q|^2)Q\right)+Q=0.
%\end{equation*}
%We have the following results that shows that  our assumptions  $\mathcal{G}(u)>0$ and $\mathcal E(u)<\gamma(c)$ are equivalent to \eqref{HR1} and \eqref{HR2} in the context of focusing and cubic NLS.

\begin{prop}\label{prop:coinc}
Under the hypothesis of \autoref{thm:1} the initial datum $u_0$ satisfies 
\begin{equation}\label{HR3}
\mathcal M(u_0)\mathcal{E}(u_0)<\mathcal M(Q)\mathcal{E}(Q)
\end{equation}
and 
\begin{equation}\label{HR4}
\|u_0\|_{L^2}\|\nabla u_0\|_{L^2}<\| Q\|_{L^2}\|\nabla Q\|_{L^2}.
\end{equation}

\noindent Moreover, with an analogous reasoning, conditions expressed in \eqref{HR3} and \eqref{HR4} imply that the initial datum falls into the hypothesis of  \autoref{thm:1}.
\end{prop}
\begin{proof}
From the definition of the quantities in \eqref{energy}, \eqref{kinetic:en} and   \eqref{potential:en} related to \eqref{GP}, we straightforwardly have
\begin{equation*} 
\mathcal E(u_0)-\frac 13 \mathcal G(u_0)=\frac 16 \mathcal T(u_0).
\end{equation*}

\noindent We notice that 
$Q_{\mu}=\mu Q(\mu x)$ is again a minimizer for the Weinstein functional with
\begin{equation*}
\begin{aligned}
\|Q_{\mu}\|_{L^2}^2&=\mu^{-1}\|Q\|_{L^2}^2, \\
\|\nabla Q_{\mu}\|_{L^2}^2&=\mu\|\nabla Q\|_{L^2}^2.
\end{aligned}
\end{equation*}
We notice that $Q(x)e^{i t}$ is a standing wave solution to the evolution equation and by the symmetry
of the equation it is well known that $Q_{\mu}e^{i \mu^2t}=\mu Q(\mu x)e^{i \mu^2t}$ is another   standing wave solution to
\begin{equation*}
-\frac12\Delta Q_{\mu}+\left(\lambda_1|Q_{\mu}|^2Q_{\mu}+\lambda_2(K\ast|Q_{\mu}|^2)Q_{\mu}\right)+\mu^2Q_{\mu}=0,
\end{equation*}
that necessarily satisfies
$\mathcal G(Q_{\mu})=0$. Hence 
$\mathcal E(Q_{\mu})=\frac 16  \|\nabla Q_{\mu}\|_{L^2}^2$. From the condition $\mathcal E(u)<\gamma(c)=\mathcal E(Q_{\mu})$ with $c=\|u\|_{L^2}^2=\|Q_{\mu}\|_{L^2}^2,$   we get
\begin{equation*}
\|u\|_{L^2}^2 \mathcal E(u)<\|Q\|_{L^2}^2\mathcal E(Q),   
\end{equation*}
which corresponds to \eqref{HR3}. Moreover if $\mathcal G(u)>0$ and $\mathcal E(u)<\gamma(c)=\mathcal E(Q_{\mu})$, then we have
\begin{equation*}
\frac 16  \|\nabla Q_{\mu}\|_{L^2}^2=\mathcal E(Q_{\mu})>\mathcal E(u)>\mathcal E(u)-\frac 13 \mathcal G(u)=\frac 16 \|\nabla u\|_{L^2}^2
\end{equation*}
and hence
\begin{equation*}
\|u\|_{L^2}\|\nabla u\|_{ L^2}< \|Q\|_{L^2}\|\nabla Q\|_{L^2}.
\end{equation*}
%which is the analogous of condition  \eqref{HR2} in our setting.
\end{proof}
\begin{remark}
It is worth doing a brief parallelism between the Cauchy problem for the Gross-Pitaevkii equation \eqref{GP} and the Cauchy problem  for the focusing cubic NLS in three dimension
\begin{equation}\label{NLS}
\left\{ \begin{aligned}
i\partial_{t}w+\frac12\Delta w&=-|w|^{2}w, \quad (t,x)\in \R\times 
\mathbb{R}^3\\
w(0,x)&=w_0(x)\in H^1
\end{aligned}\right..
\end{equation}
In \cite{HR}  sufficient conditions for global existence (and scattering) for \eqref{NLS} have been shown. They are given in terms of the energy and the mass of initial data with respect to the same quantities associated to the ground state $S$ for \eqref{NLS}, the latter being the solution to 
\begin{equation*}
-\frac12\Delta S + S =|S|^2S.
\end{equation*}
The two conditions are precisely 
\begin{equation}\label{HR1}
M(w_0)E(w_0)<M(S)E(S)
\end{equation}
and 
\begin{equation}\label{HR2}
\|w_0\|_{L^2}\|\nabla w_0\|_{L^2}<\| S\|_{L^2}\| \nabla S\|_{L^2},
\end{equation}
where 
\begin{equation*}
M(w)=\|w\|_{L^2}^2, \quad E(w)=\frac12\int |\nabla w(t)|^2-|w(t)|^4\,dx
\end{equation*}
are conserved along the NLS flow (notice that they are exactly the analogous quantities defined in \eqref{eq:mass} and \eqref{energy} for  \eqref{GP}).
Conditions \eqref{HR1} and \eqref{HR2} ensure a uniform bound on the $H^1$ norm of the solution $w(t)$ to \eqref{NLS} on its lifespan, hence it exists for every time according to the well-known blow-up alternative criterium. Therefore our assumptions  $\mathcal{G}(u)>0$ and $\mathcal E(u)<\gamma(c)$ play the same role for \eqref{GP} as \eqref{HR1} and \eqref{HR2} in the context of the cubic focusing NLS.
\end{remark}

We pass now to understand the geometry of $\mathcal{E}(u)$ on $S(c),$ and with this aim we introduce the scaling
\begin{equation*}
u^\mu(x)=\mu^{3/2}u(\mu x), \quad \mu>0.
\end{equation*}
The next lemma is contained in \cite{BJ}.
\begin{lemma}\label{lem:growth}{\cite[Lemma 3.3]{BJ}}
Let $u \in S(c)$ be such that $ \int (\lambda_1+\lambda_2 \hat K(\xi))(\widehat{ |u|^2})^2\,d\xi <0$ then we have: 
\begin{itemize}
\item[(1)] there exists a unique $\mu^{\star}(u)>0$, such that $u^{\mu^{\star}} \in V(c)$ (defined in \eqref{def:vc});
\item[(2)] the map $\mu \mapsto \mathcal{E}(u^{\mu})$ is concave on $[\mu ^{\star}, \infty)$;
\item[(3)] $\mu^{\star}(u)<1$ if and only if $\mathcal{G}(u)<0$;
\item[(4)] $\mu^{\star}(u)=1$ if and only if $\mathcal{G}(u)=0$;
\item[(5)] the functional $\mathcal G$ satisfies 
\begin{equation*}
\mathcal{G}(u^\mu)
\begin{cases}
 >0,\quad \forall\, \mu \in (0,\mu^\star(u))\\
 <0, \quad \forall\, \mu\in (\mu^\star(u),+\infty)
\end{cases};
\end{equation*}
\item[(6)] $\mathcal{E}(u^{\mu})<\mathcal{E}(u^{\mu^{\star}})$, for any $\mu>0$ and $\mu \neq \mu^{\star}$;
\item[(7)] $\frac{\partial}{\partial \mu} \mathcal{E}(u^{\mu})=\frac{1}{\mu}\mathcal{G}(u^{\mu})$, $\forall \mu>0$.
\end{itemize}
\end{lemma}
\begin{proof}
Since \begin{equation}\label{energy:mu}
\mathcal{E}(u^{\mu})=\frac{\mu^2}{2}\mathcal{T}(u)+\frac{\mu^3}{2}\mathcal{P}(u)
\end{equation}
we have that
\begin{equation*}
 \frac{\partial}{\partial \mu} \mathcal{E}(u^{\mu}) = \mu \mathcal{T}(u)+\frac{3}{2}\mu^{2}\mathcal{P}(u)
= \frac{1}{\mu}\mathcal{G}(u^{\mu}).
\end{equation*}
Now we denote
\begin{equation*}
y(\mu)= \mu \mathcal{T}(u) + \frac{3}{2}\mu^2\mathcal{P}(u),
\end{equation*}
and we observe that $\mathcal{G}(u^{\mu})= \mu  y(\mu)$ which proves (7). After direct calculations, we  see that:
\begin{equation*}
\begin{aligned}
y^\prime(\mu)&= \mathcal{T}(u) +3\mu\mathcal{P}(u), \\
y^{\prime\prime}(\mu)&= 3\mathcal{P}(u).
\end{aligned}
\end{equation*}
\noindent From the expression of $y^\prime(\mu)$ and the assumption $\mathcal{P}(u)<0$ we know that $y^\prime(\mu)$ has a unique zero that we denote $\mu_0>0$ such that $\mu_0$ is the unique maximum point of $y(\mu)$. Thus in particular the function $y(\mu)$ satisfies:\begin{itemize}
\item[(i)] $y(\mu_0)=\max_{\mu>0}y(\mu);$
\item[(ii)] $\lim_{\mu\to +\infty}y(\mu)=-\infty;$
\item[(iii)] $y(\mu)$ decreases strictly in $[\mu_0, +\infty)$ and increases strictly in $(0, \mu_0]$.
\end{itemize}

\noindent By the continuity of $y(\mu)$, we deduce that $y(\mu)$ has a unique zero $\mu^{\star}>0$. Then $\mathcal{G}(u^{\mu^\star}) =0$ and point (1) follows. Points (2)--(5) are also easy consequences of (i)--(iii). 
Finally, since $y(\mu) >0$ on $(0, \mu^\star(u))$ and $y(\mu) <0$ on $(\mu^\star(u), \infty)$ we get (6). 
\end{proof} 

We are now in the position to state a lower bound for \eqref{GG}.
\begin{prop}\label{cor:lb} Under the hypothesis of \autoref{thm:1}
\begin{equation*}
4\int|\nabla u|^2\,dx+6\lambda_1\int|u|^4\,dx+6\lambda_2\int(K\ast|u|^2)|u|^2\,dx\geq 4\min\{\gamma(c)-\mathcal{E}(u_0), \mathcal{E}(u_0)\}:=\alpha.
\end{equation*}
\end{prop}
\begin{proof}
We shall distinguish two cases which depend on the sign of the nonlinear term. At a given time $\tilde t$ we have either $\mathcal{P}(u(\tilde t))>0$ or $\mathcal{P}(u(\tilde t))<0$. Let $\tilde t$ therefore be an arbitrary but fixed time and simply denote $u = u(x,\tilde t)$. We consider both cases:\\

\noindent \emph{Case 1: $\mathcal{P}(u)>0.$} In this case the estimate is trivial. Indeed, $\mathcal{G}(u)=2\mathcal{E}(u)+\frac 12 \mathcal{P}(u)>2\mathcal{E}(u)>\mathcal{E}(u).$\\

\noindent \emph{Case 2: $\mathcal{P}(u)<0.$}
In this case we argue using the scaling of  \autoref{lem:growth}.  First let us notice that if $\left|\frac{\mathcal{P}(u)}{2}\right|< \mathcal{E}(u)$ then $\mathcal{G}(u)=2 \mathcal{E}(u)+\frac{\mathcal{P}(u)}{2}>\mathcal{E}(u)$ then the lower bound is achieved. We assume hence that $\left|\frac{\mathcal{P}(u)}{2}\right|\geq  \mathcal{E}(u)$.  Let us rescale the function $u$ according to the scaling of  \autoref{lem:growth} such that $u^{\mu^{\star}} \in V(c)$ and let us express $\mathcal{E}(u^{\mu^{\star}})-\mathcal{E}(u) $ as
\begin{equation*}
\mathcal{E}(u^{\mu^{\star}})-\mathcal{E}(u) = (\mu^{\star}-1) \frac{\partial}{\partial \mu}\mathcal{E}(u^{\mu})\mid _{\mu =\mu_0} 
\end{equation*}
with $1<\mu_0<\mu^{\star},$ according to the previous Lemma. Now we claim  that 
\begin{equation}\label{eq:claim}
\mu^{\star}<2 \quad\hbox{ and }\quad  \frac{\partial}{\partial \mu}\mathcal{E}(u^{\mu})\mid _{\mu =\mu_0} <\frac{\partial}{\partial \mu}\mathcal{E}(u^{\mu})\mid _{\mu =1}= \mathcal{G}(u).
\end{equation}
From the claim we get the desired estimate. Indeed,
\begin{equation*}
\gamma(c)-\mathcal{E}(u_0 )\leq \mathcal{E}(u^{\mu^{\star}})-\mathcal{E}(u)\leq \mathcal{G}(u),
\end{equation*}
and this concludes the proof. 

\noindent Thus, it remains to prove  \eqref{eq:claim}. From \eqref{energy:mu} we get that $\mathcal{E}(u^{\mu})<0$ for $\mu>\frac{\mathcal{T}(u)}{|\mathcal{P}(u)|},$ hence
$\mu^{\star}<\frac{\mathcal{T}(u)}{|\mathcal{P}(u)|}$. Now we notice that
the assumption $\frac12\left| P(u)\right|\geq  \mathcal{E}(u)$ implies that  $\frac{\mathcal{T}(u)}{|\mathcal{P}(u)|}=\frac{2\mathcal{E}(u)+|\mathcal{P}(u)|}{|\mathcal{P}(u)|}\leq2$. To prove  $\frac{\partial}{\partial \mu}\mathcal{E}(u^{\mu})\mid _{\mu =\mu_0} < \mathcal{G}(u)$ 
it is sufficient to show that $\frac{\partial}{\partial \mu}\mathcal{E}(u^{\mu})$ is monotone decreasing when $\mu>1$. Direct computation gives in fact
\begin{equation*}
\frac{\partial^2}{\partial^2 \mu}\mathcal{E}(u^{\mu})= \mathcal{T}(u)+3 \mu \mathcal{P}(u)<0
\end{equation*}
provided that $\mu>\frac{\mathcal{T}(u)}{3 |\mathcal{P}(u)|}=\frac{2\mathcal{E}(u)+|\mathcal{P}(u)|}{3|\mathcal{P}(u)|}$. Now the condition $\frac12\left| P(u)\right|\geq  \mathcal{E}(u)$ implies that \begin{equation*}
\frac{2\mathcal{E}(u)+|\mathcal{P}(u)|}{3|\mathcal{P}(u)|}\leq\frac23.
\end{equation*}
The fact that $\mu_0>1$ proves that claim. Summing up all the estimates we get
\begin{equation*}\mathcal{G}(u)>\min\{\gamma(c)-\mathcal{E}(u_0), \mathcal{E}(u_0)\}.
\end{equation*}
\end{proof}

\section{Small data theory and perturbative nonlinear results}\label{perturbative}
We collect here some perturbative results for \eqref{GP}. To lighten the exposition, we  just give here the statements, postponing the proofs until \autoref{app}.
The next two Lemmas are in the framework of the so-called \emph{Small data Theory,} which is actually the first cornerstone on which the Kenig and Merle approach is built. The first one ensures that if the initial datum is sufficiently small in the $H^1$-norm, then its nonlinear evolution under the Gross-Pitaevskii flow \eqref{GP} is global. 
\begin{lemma}\label{lemma3.1}
There exists a radius $\rho>0$ such that if $u_0\in B_{\rho}(H^1)$ then the corresponding solution $u\in\mathcal C(-T_{min},T_{max});H^1)$ to \eqref{GP} with $u_0$ as initial datum is global, i.e. $T_{min}=T_{max}=+\infty.$ 
\end{lemma}
\noindent The second one claims that if the initial datum is small enough (possibly smaller than the previous one), still in the $H^1$ space, then the global solution to \eqref{GP} actually behaves like a free wave asymptotically in time. 
\begin{lemma}\label{lemma3.2}
There exists $\delta>0$ such that for any $u_0\in B_{\delta}(H^1),$ the solution $u(t,x)$ to the Cauchy problem \eqref{GP} 
 {\em scatters} to a linear solution in $H^1$.
\end{lemma}
\noindent The following states that if a global solution to \eqref{GP} enjoys some uniform spacetime control, then it scatters.
\begin{lemma}\label{lemma3.3}
If $u(t,x)\in\mathcal{C}(\mathbb{R};H^1)\cap L^8L^4$ is a solution to \eqref{GP}, then it scatters.
\end{lemma}
\begin{remark}
It is worth mentioning that if a $\mathcal{C}(\mathbb{R};H^1)$ solution to \eqref{GP} scatters, then it belongs to $ L^8L^4.$
\end{remark}

\noindent We conclude this preliminary tools section with a small perturbation result.
\begin{lemma}\label{lemma3}
For every $M>0$ there exist $\varepsilon=\varepsilon(M)>0$ and $C=C(M)>0$ such that: if $u(t,x)\in\mathcal{C}(\mathbb{R};H^1)$ is the unique global solution to \eqref{GP} and $w\in\mathcal{C}(\mathbb{R};H^1)\cap L^8L^4$ is a global solution to the perturbed problem 
\begin{equation*}
\left\{ \begin{aligned}
i\partial_{t}w+\frac12\Delta w&=\lambda_1|w|^{2}w+\lambda_2(K\ast|w|^2)w+e(t,x)\\
w(0,x)&=w_0\in H^1
\end{aligned} \right. 
\end{equation*}
satisfying the conditions $\|w\|_{L^8L^4}\leq M,$
$\|\int_{t_0}^{t} U(t-s)e(s)\,ds\|_{L^8L^4}\lesssim\varepsilon$ and
$\|U(t-t_0)(u(t_0)-w(t_0))\|_{L^8L^4}\leq\varepsilon,$ 
then $u\in L^8L^4$ and $\|u-w\|_{L^8L^4}\lesssim\varepsilon.$
\end{lemma}
\begin{remark}\label{rem:nlp}
%The profile decomposition theorem below is the fundamental tool to construct a minimal (with respect to the energy) global but not scattering solution to \eqref{GP}. As il will be clear in the next section, that is a linear result. Therefore the perturbation lemma above will be crucial to absorb the error once some nonlinear profiles are associated to the linear ones of \autoref{thm:pd}. We refer the reader to the next sections for more rigorous statements. 

Let us mention that the term $e(t,x)$ in \autoref{lemma3} is meant as an error. The main tool in the construction of a minimal (with respect to the energy) global but not scattering solution to \eqref{GP} which will be carried out in \autoref{min:el}, will be the profile decomposition theorem contained in the Section below. It is worth mentioning since now on that the profile decomposition theorem is a purely linear statement, though we are concerned with the construction of a nonlinear solitone-like solution. Therefore \autoref{lemma3} will be crucial to absorb the error terms that will be introduced once one passes from the linear profiles given by \autoref{thm:pd} to nonlinear profiles.
\end{remark}
\section{Linear Profile Decomposition and Nonlinear Profiles}\label{profile}
The main ingredient in the construction of the minimal element is a suitable profile decomposition theorem. It is worth mentioning that this kind of result goes back to the work of G\'erard, see \cite{Ger}, where is given an explicit characterization of the defect of compactness for the Sobolev embeddings. Pioneering results for evolution equations are the works by Bahouri and G\'erard, see \cite{BG}, for the critical wave equation, by Keraani, see \cite{Ker}, about the defect of compactness for the Strichartz embeddings, and by Merle and Vega, see \cite{MV}, how to treat concentration phenomena for the two dimensional mass critical nonlinear Schr\"odinger equation.

If it were not for the presence of the nonlocal term, the next linear result in the non-radial setting would be exactly the same given by Duyckaerts, Holmer and Roudenko in \cite{DHR}, which in turn extended the one in Holmer and Roudenko \cite{HR} removing the spherical symmetry assumption of the this last mentioned paper. For \eqref{GP}, an additional term must be dealt with, so at first we state the theorem in \cite{DHR}, then we will show how to manage with the nonlocal term of \eqref{GP}.
\begin{theorem}\label{thm:pd}  Given a bounded sequence $\{v_n\}_{n\in\mathbb{N}}\subset H^1,$ $\forall\, J\in\mathbb{N}$ and  $\forall\,1\leq j\leq J$ there exist sequences of time and space translation parameters $\{t_n^j\}_{n\in\mathbb{N}}\subset\R,\,\{x_n^j\}_{n\in\mathbb{N}}\subset\R^3$ and profiles $\psi^j\in H^1$ for $j=1, \dots,J$ and $R^J_n$ such that, up to subsequences,
\begin{equation}\label{deco2}
v_n=\sum_{1\leq j\leq J}U(-t_n^j)\tau_{x_n^j}\psi^j+R_n^J
\end{equation} 
with the following properties:
\begin{itemize}
\item{[Dichotomy of the parameters]} for any fixed $j\in\{1,\dots,J\}$
\begin{align*}
either\quad t_n^j=0\quad \forall\, n\in\mathbb{N}\quad &or \quad t_n^j\ra\pm\infty,\\
either\quad x_n^j=0\quad \forall\, n\in\mathbb{N}\quad &or \quad x_n^j\ra\pm\infty;
\end{align*} 
\item{[Divergence property]} for any $j\neq k\in\{1,\dots,J\}$
\begin{equation*}
|x_n^j-x_n^k|+|t_n^j-t_n^k| \overset{n\ra\infty}\longrightarrow\infty;
\end{equation*}
\item{[Smallness of the remainder]} $\forall\,\varepsilon>0\quad\exists\,\tilde J=\tilde J(\varepsilon)$ such that, for any $J\geq\tilde J$
\begin{equation}\label{small:rem} 
\limsup_{n\ra\infty}\|U(t)R_n^J\|_{L^\infty L^3\cap L^8L^4}\leq\varepsilon;
\end{equation}
\item{[Pythagorean expansion of  mass and kinetic energy]} the mass and the quadratic energy term are almost orthogonal, namely as $n\to\infty$
\begin{align}\label{mass:dec}
\|v_n\|_{L^2}^2=&\sum_{1\leq j\leq J}\|\psi^j\|_{L^2}^2+\|R_n^J\|_{L^2}^2+o(1), \quad \forall\, J\in\mathbb{N},\\\label{kinetic:dec}
\|v_n\|^2_{H^1}=&\sum_{1\leq j\leq J}\|\psi^j\|^2_{H^1}+\|R_n^J\|^2_{H^1}+o(1), \quad \forall\, J\in\mathbb{N};
\end{align}
\item{[Pythagorean expansion of the local potential energy]} $\forall\, J\in\mathbb{N},$ as $n\to\infty$
\begin{equation}\label{energy:dec1}
\|v_n\|_{L^4}^4=\sum_{1\leq j\leq J} \|U(-t_n^j) \tau_{x_n^j}\psi^j\|_{L^4}^4+\|R_n^J\|_{L^4}^4+o(1).
\end{equation}
\end{itemize}
\end{theorem}
\begin{remark}
Let us briefly comment that the dichotomy of the parameters comes from the fact that if the sequence $\{t_n\}_{n\in\N}$ is bounded, then up to subsequences it converges to some $\bar t$ and similarly can be claimed for the space translation sequence $\{x_n\}_{n\in\N}.$ Therefore, in case the sequences of parameters are bounded, up to translate in time and space the profiles $\psi^j$, they can be assumed to be the trivial ones. The divergence property instead implies that the profiles weakly interact each other, leading to the orthogonality relations in the theorem.
\end{remark}
If the nonlocal term in \eqref{GP} were not present, namely $\lambda_2=0,$ summing up \eqref{kinetic:dec} and \eqref{energy:dec1} would lead to the so called orthogonal decomposition of the energy. But in our case we must deal with the nonlocal interaction term $\lambda_2\int (K\ast|v|^2)|v|^2.$ The next proposition aims to show exactly that what is inferred in  \eqref{energy:dec1} has its counterpart also for the dipolar interaction energy. 
\begin{prop}\label{pd2} Under the same hypothesis of \autoref{thm:pd}, the following Orthogonal Expansion of the nonlocal energy can be claimed:  for any $J\in\N,$ as $n\to\infty$
\begin{equation}\label{nonlocal:dec}
\begin{aligned}
\int (K\ast|v_n|^2)|v_n|^2\,dx&=\sum_{1\leq j\leq J} \int (K\ast|U(-t_n^j) \tau_{x_n^j}\psi^j|^2)|U(-t_n^j) \tau_{x_n^j}\psi^j|^2\,dx\\
&+\int (K\ast|R^J_n|^2)|R^J_n|^2\,dx+o(1).
\end{aligned}
\end{equation}
\end{prop}
\begin{proof} 
Up to reordering the indexes, we may suppose that there exists $\bar J$ such that: 
\begin{itemize}
\item[\emph{Case 1}:] $t^j_n=0$ for any $n\in\N,$ if $1\leq j\leq\bar J;$ 
\item[\emph{Case 2}:] $|t^j_n|\to\infty$ as $n\to\infty$ if $\bar J+1\leq j\leq J.$
\end{itemize}
\noindent Due to the divergence property in \autoref{thm:pd}, in the situation of \emph{Case 1}, given two different indexes $j\neq k\in\left\{1,\dots,\bar J\right\}$ then $|x^j_n-x^k_n|\to\infty$ as $n\to\infty$ and this implies the weak interaction for the cross term in the expression below:
\begin{equation*}
\begin{aligned}
&\iint K(x-y)\left|\sum_{j=1}^{\bar J}U(-t^j_n)\psi^j(y-x^j_n)\right|^2\left|\sum_{j=1}^{\bar J}U(-t^j_n)\psi^j(x-x^j_n)\right|^2\,dy\,dx\\
=&\iint K(x-y)\left|\sum_{j=1}^{\bar J}\psi^j(y-x^j_n)\right|^2\left|\sum_{j=1}^{\bar J}\psi^j(x-x^j_n)\right|^2\,dy\,dx.
\end{aligned}
\end{equation*}
More precisely, since $K\ast\tau_zg=\tau_z(K\ast g),$ then as $|z-z^\prime|\to\infty$
\begin{equation}\label{cross}
\int (K\ast\tau_zg)(x)\tau_{z^\prime}h(x)\,dx=\int \tau_z(K\ast g)(x)\tau_{z^\prime}h(x)\,dx=o(1).
\end{equation}
Hence \eqref{cross} implies that in the situation delineated in  \emph{Case 1} 
\begin{equation}\label{cross3}
\begin{aligned}
&\iint K(x-y)\left|\sum_{j=1}^{\bar J}\psi^j(y-x^j_n)\right|^2\left|\sum_{j=1}^{\bar J}\psi^j(y-x^j_n)\right|^2\,dy\,dx\\
=&\sum_{j=1}^{\bar J}\iint K(x-y)\left|\psi^j(y-x^j_n)\right|^2\left|\psi^j(y-x^j_n)\right|^2\,dy\,dx\\
=&\sum_{j=1}^{\bar J}\iint K(x-y)\left|U(-t^j_n)\psi^j(y-x^j_n)\right|^2\left|U(-t^j_n)\psi^j(y-x^j_n)\right|^2\,dy\,dx.
\end{aligned}
\end{equation}
Under the condition illustrated in \emph{Case 2} instead, the continuity property of the operator $K\ast f,$ mapping continuously $L^p$ into itself for any $p\in(1,\infty),$ yields, by using the Cauchy-Schwarz inequality and the continuity property with $p=2,$
\begin{equation}\label{cross4}
\begin{aligned}
&\int (K\ast(|U(-t^j_n)\psi^j(\cdot-x^j_n)|^2))(x)|U(-t^j_n)\psi^j(\cdot-x^j_n)|^2(x)\,dx\\
&\lesssim\|K\ast(|U(-t^j_n)\psi^j(\cdot-x^j_n)|^2)\|_{L^2}\||U(-t^j_n)\psi^j(\cdot-x^j_n)|^2\|_{L^2}\\
&\lesssim \|U(-t^j_n)\psi^j\|_{L^4}^4\overset{n\to\infty}{\longrightarrow}0.
\end{aligned}
\end{equation}
The last decay property follows by the dispersive estimate \eqref{disp:est2} of the Schr\"odinger free propagator and the fact that $U(t)$ is an isometry on $L^2$ (actually, on any $H^s$ space), concluding with a density argument by considering at first $\psi^j\in L^{4/3}\cap H^1.$ Summing up the results in \eqref{cross3} and \eqref{cross4} we claim that for $n\to\infty$
\begin{equation*}
\begin{aligned}
&\iint K(x-y)\left|\sum_{j=1}^{J}U(-t^j_n)\psi^j(y-x^j_n)\right|^2\left|\sum_{j=1}^{J}U(-t^j_n)\psi^j(x-x^j_n)\right|^2\,dy\,dx\\
&=\sum_{j=1}^{ J}\iint K(x-y)\left|U(-t^j_n)\psi^j(y-x^j_n)\right|^2\left|U(-t^j_n)\psi^j(x-x^j_n)\right|^2\,dy\,dx+o(1).
\end{aligned}
\end{equation*}
Recall from \eqref{nonlocal:dec} that we aim to prove that the quantity
\begin{equation}\label{cross5}
\begin{aligned}
\mathcal I:=&\int (K\ast|v_n|^2)|v_n|^2\,dx-\sum_{j=1}^{ J}\int (K\ast|U(-t^j_n)\tau_{x^j_n}\psi^j|^2)|U(-t^j_n)\tau_{x^j_n}\psi^j|^2\,dx-\int (K\ast|R^J_n|^2)|R_n^J|^2\,dx\\
=&\int (K\ast|v_n|^2)|v_n|^2\,dx-\sum_{j=1}^{ J}\int (K\ast|U(-t^j_n)\tau_{x^j_n}\psi^j|^2)|U(-t^j_n)\tau_{x^j_n}\psi^j|^2\,dx-\int (K\ast|R^J_n|^2)|R_n^J|^2\,dx\\
\pm&\int (K\ast|v_n-R^L_n|^2)|v_n-R^L_n|^2\,dx\pm\int (K\ast|R^J_n-R^L_n|^2)|R^J_n-R^L_n|^2\,dx
\end{aligned}
\end{equation}
goes to zero as $n\to\infty,$ where $L$ is a fixed positive integer. To shorten the notation we define
\begin{equation*}
\begin{cases}
g_n^L=v_n-R^L_n\\
r^{L,J}_n=R^L_n-R^J_n\\
u_n^j=U(-t^j_n)\tau_{x^j_n}\psi^j
\end{cases} 
\end{equation*}
therefore %the left-hand side of \eqref{cross5} can be estimated as 
$\mathcal I$ can be estimated as 
\begin{equation*}
\begin{aligned}
%left-hand side \eqref{cross5}
\mathcal I&\leq\left|\int (K\ast |v_n|^2)|v_n|^2\,dx-\int (K\ast |g_n^L|^2)|g_n^L|^2\,dx\right|\\
&+\left|\int (K\ast |r^{L,J}_n|^2)|r^{L,J}_n|^2\,dx-\int (K\ast|R^J_n|^2)|R^J_n|^2\,dx\right|\\
&+\left|\int (K\ast |g_n^L|^2)|g_n^L|^2\,dx-\sum_{j=1}^{ J}\int (K\ast |u_n^j|^2)|u_n^j|^2\,dx-\int (K\ast |r^{L,J}_n|^2)|r^{L,J}_n|^2\,dx\right|\\
&=I+II+III
\end{aligned}
\end{equation*}
and we show that this three quantities go towards zero as $n\to\infty.$ Let us begin with the first term $I.$
\begin{equation*}
\begin{aligned}
I&\leq\left|\iint K(x-y)\left( |v_n(y)|^2|v_n(x)|^2-|(g_n^L(y)|^2|g_n^L(x)|^2\right)\,dy\,dx\right|\\
&\leq \left|\iint K(x-y)\Pi^4\left(v_n(x),\bar v_n(x),v_n(y),\bar v_n(y),R^L_n(x),\bar R^L_n(x),R^L_n(y),\bar R^L_n(y)\right)\,dy\,dx\right|,
\end{aligned}
\end{equation*}
where $\Pi^4\left(v_n(x),\bar v_n(x),v_n(y),\bar v_n(y),R^K_n(x),\bar R^L_n(x),R^L_n(y),\bar R^L_n(y)\right)$ is an homogeneous polynomial of order $4$ not involving any monomial consisting in only $v_n,\bar v_n$'s terms, hence it can be estimated by using the continuity property of the convolution with the kernel $K,$ obtaining 
\begin{equation*}
\begin{aligned}
I&\lesssim \|R^L_n\|_{L^4}^4+\|v_n\|_{L^4}^3\|R^L_n\|_{L^4}+\|v_n\|_{L^4}^2\|R^L_n\|_{L^4}^2\\
&\lesssim \|R^L_n\|_{L^4}^4+\|R^L_n\|_{L^4}^2+\|R^L_n\|_{L^4},
\end{aligned}
\end{equation*}
where we used that $\{v_n\}_{n\in\N}$ is uniformly bounded in $L^4$ since it is bounded in $H^1,$ by Sobolev embedding. Since 
\begin{equation*}
\begin{aligned}
\|R^L_n\|_{L^4}&\leq \|U(t)R^L_n\|_{L^\infty L^4}\leq \||U(t)R^L_n|^{1/2}\|_{L^\infty L^{12}}\||U(t)R^L_n|^{1/2}\|_{L^\infty L^6}\\
&\leq \|U(t)R^L_n\|^{1/2}_{L^\infty L^3}\|U(t)R^L_n\|_{L^\infty L^6}^{1/2}
\end{aligned}
\end{equation*}
combining \eqref{small:rem} and \eqref{kinetic:dec} one obtains 
\begin{equation*}
\limsup_{J\to\infty}\lim_{n\to\infty}\|R^J_n\|_{L^\infty L^4}=0.
\end{equation*}

\noindent Hence, for any $\varepsilon>0$ there exists $L$ large enough and $\bar n$ such that for any $n>\bar n$ the term $I\leq\varepsilon/3.$ The same analysis can be carried out for $II,$ therefore also $II\leq\varepsilon/3.$ 
Let turn the attention on the last term 
\begin{equation*}
III=\left|\int (K\ast| g_n^L|^2)|g_n^L|^2\,dx-\sum_{j=1}^{ J}\int (K\ast |u_n^j|^2)|u_n^j|^2\,dx-\int (K\ast|r^{L,J}_n|^2)|r^{L,J}_n|^2\,dx\right|.
\end{equation*}
By definition 
\begin{equation*}
\begin{cases}
g_n^L=v_n-R^L_n=\sum_{j=1}^Lu^j_n\\
g_n^J=v_n-R^J_n=\sum_{j=1}^Ju^j_n\\
R^L_n-R^J_n=\sum_{j=L+1}^Ju^j_n
\end{cases},
\end{equation*}
then
\begin{equation*}
\begin{aligned}
III&\leq\left|\iint K(x-y)\left|\sum_{j=1}^Lu^j_n(y)\right|^2\left|\sum_{j=1}^Lu^j_n(x)\right|^2\,dy\,dx-\sum_{j=1}^{L}\iint K(x-y) |u_n^j(y)|^2|u_n^j(x)|^2\,dy\,dx\right|\\
&+\left|\iint K(x-y)\left|\sum_{j=L+1}^Ju^j_n(y)\right|^2\left|\sum_{j=L+1}^Ju^j_n(x)\right|^2\,dy\,dx-\sum_{j=L+1}^{J}\iint K(x-y) |u_n^j(y)|^2|u_n^j(x)|^2\,dy\,dx\right|\\
&\leq \varepsilon/6+\varepsilon/6=\varepsilon/3.
\end{aligned}
\end{equation*}
Hence \eqref{nonlocal:dec} is proved.
\end{proof}
As immediate corollary we have:
\begin{corollary}\label{cor:dec}
Under the hypothesis of \autoref{thm:pd}, for the decomposition in \eqref{deco2} the Pythagorean expansion of the energy holds true, namely: $\forall J\in N$ as $n\to+\infty$
\begin{equation*}
\mathcal E(v_n)=\sum_{1\leq j\leq J} \mathcal E(U(-t_n^j) \tau_{x_n^j}\psi^j)+\mathcal E(R_n^J)+o(1).
\end{equation*}
\end{corollary}
As already pointed out in \autoref{rem:nlp}, we will need to associate nonlinear profiles to the linear ones of \autoref{thm:pd}. The existence of such nonlinear waves basically comes from the local well-posedness theory for \eqref{GP} and the existence of wave operators. More precisely, once we consider a pair $(\psi^j, t^j_n)$ as above, a nonlinear profile corresponding to it is given by means of a solution $\Psi^j$ to \eqref{GP} satisfying 
\begin{equation*}
\|\Psi^j(t^j_n)-U(t^j_n)\psi^j\|_{H^1}\overset{n\to\infty}\longrightarrow 0.
\end{equation*}
Since the dichotomy property of the parameters in \autoref{thm:pd} allows us to restrict the situation only on the cases $t^j_n=0$ or $t^j_n\to\infty,$ finding such $\Psi^j$ reduces to solving the Cauchy problem for the GPE at $t_0=0$ or at $t_0=\infty,$ namely it suffices to solve 
\begin{equation*}
\Psi^j=U(t)\psi^j+i\int_{0}^tU(t-s)\left(\lambda_1|\Psi^j|^2\Psi^j+\lambda_2(K\ast|\Psi^j|^2)\Psi^j\right)(s)\,ds
\end{equation*} 
and
\begin{equation*}
\Psi^j=U(t)\psi^j+i\int_{\infty}^tU(t-s)\left(\lambda_1|\Psi^j|^2\Psi^j+\lambda_2(K\ast|\Psi^j|^2)\Psi^j\right)(s)\,ds,
\end{equation*} 
where the equations above are referred as the integral formulation of the solution to \eqref{GP},  also known as Duhamel's representation.
\section{Existence of the Minimal Element and its properties}\label{min:el}
Once the profile decomposition is proved, combinations of arguments  in \autoref{sec:var:est} and  \autoref{perturbative} give the following.
\begin{theorem}\label{lemcri}
There exists a not trivial initial profile $u_{sl,0}\in H^1$
with $\mathcal E(u_{sl,0})<\gamma(\|u_{sl,0}\|_{L^2}^2)$ and $\mathcal G(u_{sl,0})>0$ such that  the corresponding solution $u_{sl}$ to \eqref{GP} is globally defined but  does not scatter. Moreover, there exists a continuous function $x(t) : \R^+\mapsto\R^3$ such that $\{u_{sl}(t,x-x(t)),\,t\in\R^+\}$ is precompact as a subset of $H^1.$
\end{theorem}
\begin{definition}
It is standard to refer to this solution as the \emph{minimal element} or \emph{critical solution} or \emph{soliton-like solution}. Along the remaining part of the paper we follow this conventions and we will denote it as $u_{sl},$ justifying therefore the subscripts in \autoref{lemcri}.
\end{definition}
\begin{proof}
We just sketch the proof which relies on a contradiction argument. Let us define the threshold
\begin{align*}
\gamma_{sl}(c)=\sup\left\{ \right. &\gamma>0 \textit{ such that if } u_0\in S(c)\, \textit{ with } \mathcal E(u_0)<\mathcal\gamma \textit{ and } \mathcal G(u_0)>0\\\notag
&\qquad\,\textit{then the solution to \eqref{GP} with initial data } u_0 \textit{ is in } L^8L^4\left.\right\}.
\end{align*}
Since small data scattering holds, see \autoref{lemma3.1} and \autoref{lemma3.2}, the set above is well-defined and  $\gamma_c>0.$  

The aim is to show that $\gamma_{sl}(c)=\gamma(c)$ and we suppose that this is not the case, namely we assume that $\gamma_{sl}(c)<\gamma(c).$
Consider a minimizing sequence of initial data $\{u_n(0)\}_{n\in\N}$ with $\mathcal E(u_n(0))\downarrow \gamma_{sl}$ and such that the corresponding sequence of solutions $\{u_n(t)\}_{n\in\N}$ to \eqref{GP} satisfy 
\begin{equation}\label{contra:hyp}
\limsup_{n\to\infty}\|u_n\|_{L^8L^4}\to\infty.
\end{equation}  
The sequence  $\{u_n(0)\}_{n\in\N}$ can be decomposed in
\begin{equation}\label{deco6}
u_{0,n}=\sum_{1\leq j\leq J}U(-t_n^j)\tau_{x_n^j}\psi^j+R_n^J,
\end{equation} 

\noindent by means of \autoref{thm:pd}, \autoref{pd2} and \autoref{cor:dec}. In particular,  \autoref{cor:dec} implies that, in the limit $n\to\infty,$
\begin{equation*}
\gamma_{sl}(c)=\sum_{1\leq j\leq J} \mathcal E(U(-t_n^j)\psi^j)+\mathcal E(R_n^J)
\end{equation*}
while the orthogonal expansion of the mass \eqref{mass:dec} gives, as $n\to\infty,$ (since $\|u_{0,n}\|_{L^2}^2=c$ for any $n\in\N$)
\begin{equation*}
c^j:=\|\psi^j\|_{L^2}^2\leq c.
\end{equation*}

We claim the following: there exists only one non trivial profile in the expansion \eqref{deco6}.\\

\noindent Suppose by the absurd that at least two profiles are non-trivial,  i.e. $\psi^{j_h}\neq0$ for $\{j_h\}\subseteq\{1,\dots,J\}$ and the cardinality $\#\{j_h\}\geq2.$ Let us keep the notation $\psi^j$ instead of $\psi^{j_h}.$ This implies that $c^j<c$ and $\mathcal E(U(-t_n^j)\psi^j)<\gamma_{sl}(c).$

\noindent We recall that the equation \eqref{GP} is invariant under the transformation $u\mapsto u_{\mu}:=\mu u(\mu^2 t,\mu x),$ the latter moreover leaving invariant the $L^8L^4$-norm, and we split the situation in two cases. \\

\emph{Case $1$.}  In this first situation we assume that the time translation parameter of the profile decomposition above is diverging, namely $t^j_n\to\infty$ for some $j.$  In this case we have that $\lim_{n \rightarrow \infty}\mathcal{G}(U(-t_n^j)\tau_{x_n^j}\psi^j)>0$ and $\lim_{n \rightarrow \infty}\mathcal{E}(U(-t_n^j)\tau_{x_n^j}\psi^j)>0$.
\noindent Since we have  $\mathcal{E}(U(-t_n^j)\psi^j)<\gamma_{sl}(c)$
and  the scaling of the equation guarantees that 
\begin{equation*}\label{cc}
c_j \gamma_{sl}(c_j)=c \gamma_{sl}(c) \implies  \gamma_{sl}(c_j)> \gamma_{sl}(c)
\end{equation*}
we can infer that $ \mathcal{E}(U(-t_n^j)\tau_{x_n^j}\psi^j)<\gamma_{sl}(c_j) .
$
With the nonlinear profiles constructed at the end of  \autoref{profile}, therefore mapping $(t^j_n,\psi^j)\mapsto \Psi^j_n,$ we get  $\mathcal{E}(\Psi^j)<\gamma_{sl}(c_j)$, $\Psi^j\in S(c_j)$, $\mathcal G (\Psi^j)>0$ and hence we obtain 
\begin{equation*}
\|\Psi^j\|_{L^8L^4}<+\infty.
\end{equation*}

\emph{Case $2$.} We consider  the remaining situation, namely when the time translation sequence is the trivial one.  

\noindent We argue as before using that the convergence 
\begin{equation*}U(-t_n^j)\tau_{x_n^j}\psi^j\rightarrow \tau_{x_n^j}\psi^j 
\end{equation*}
as $n\to\infty$  strongly holds in $H^1$ topology.\\

\noindent We first show  that $\mathcal{G}(U(-t_n^j)\tau_{x_n^j}\psi^j)>0$.
Notice that thanks to $\mathcal{G}(u_n)>0$
\begin{equation*} \frac{c_j}{6} \|\nabla U(-t_n^j)\tau_{x_n^j}\psi^j \|_{L^2}^2<  \frac c6 \|\nabla u_n \|_{L^2}^2<c \mathcal{E}(u_n)=c \gamma_{sl}(c)+o(1)=c_j \gamma_{sl}(c_j)+o(1).
\end{equation*}
Let us suppose $\mathcal{G}(U(-t_n^j)\tau_{x_n^j}\psi^j)<0$ and let us pick $0<\mu^{\star}\leq1$ such that
$\mathcal{G}(v^{\mu^\star})=0$
where $v=U(-t_n^j)\tau_{x_n^j}\psi^j$ and $v^{\mu}=\mu^{3/2}v(\mu x)$ (see  \autoref{lem:growth} for the properties of the functional $\mathcal G$ when evaluated on such rescaled functions). We have  
\begin{equation*}
\mathcal{E}(v^{\mu^\star})=\frac{\mu^{{\star}^2}}{6} \|\nabla U(-t_n^j)\tau_{x_n^j}\psi^j \|_{L^2}^2\leq \frac 16 \|\nabla U(-t_n^j)\tau_{x_n^j}\psi^j\|_{L^2}^2
\end{equation*}
and therefore 
\begin{equation*}
c_j\mathcal{E}(v^{\mu^\star})\leq   \frac{c_j}{6} \|\nabla U(-t_n^j)\tau_{x_n^j}\psi^j\|_{L^2}^2<c_j\gamma_{sl}(c_j)+o(1)<c_j\gamma(c_j)
\end{equation*}
which leads to the contradiction. \\

\noindent Now, recalling that $\mathcal{E}(U(-t_n^j)\tau_{x_n^j}\psi^j)>\frac 16 ||\nabla U(-t_n^j)\tau_{x_n^j}\psi^j||_{L^2}^2$ we get $\mathcal{E}(U(-t_n^j)\tau_{x_n^j}\psi^j)<\gamma_{sl}(c)<\gamma_{sl}(c_j).$
As for the previous case, we can associate to the linear profiles their corresponding nonlinear ones $\Psi^j$ having the property to belong to $S(c_j),$ satisfying $\mathcal G (\Psi^j)>0$ and  having finite $L^8L^4$-norm, hence they scatter. 

All ingredients to show that there can be only a nontrivial profile in the decomposition \eqref{deco6} are established, therefore the existence of the minimal element $u_{sl}$ as stated in \autoref{lemcri} can be proved following \cite{HR, DHR}. 

By substituting the linear profiles in \eqref{deco6} with the nonlinear ones coming from \emph{Case $1$} and \emph{Case $2$} above, we write 
\begin{equation*}
u_{0,n}=\sum_{1\leq j\leq J}\Psi^j(-t_n^j)+\rho_n^J,
\end{equation*} 
with 
\begin{equation}\label{small:rem:nonlin} 
\limsup_{n\ra\infty}\|U(t)\rho_n^J\|_{ L^8L^4}\leq\varepsilon
\end{equation}
for any $J\geq \bar J=\bar J(\varepsilon).$

The aim is now to give an approximation of $\{u_n(t)\}_{n\in\N}$ in terms of the scattering nonlinear profiles above to reach a contradiction by showing that $\{u_n(t)\}_{n\in\N}$ has uniformly bounded $L^8L^4$-norm by means of the perturbation result of \autoref{lemma3}. Therefore we define 
\begin{equation*}
w_n(t)=\sum_{1\leq j\leq J}\Psi^j(t-t^j_n,x-x^j_n)
\end{equation*}
and by the very definition of the involved term we get 
\begin{equation*}
i\partial_tw_n+\frac12\Delta w_n-\lambda_1|w_n|^2w_n-\lambda_2(K\ast|w_n|^2)w_n=\lambda_1|w_n|^2w_n+\lambda_2(K\ast|w_n|^2)w_n+e_n
\end{equation*}
with
\begin{equation*}
\begin{aligned}
e_n&=\sum_{1\leq j\leq J}\lambda_1\left(|\Psi^j(t-t^j_n,x-x^j_n)|^2\Psi^j(t-t^j_n,x-x^j_n)-|w_n|^2w_n\right)\\
&+\sum_{1\leq j\leq J}\lambda_2\left( \left(K\ast|\Psi^j(t-t^j_n,\cdot-x^j_n)|^2\right)\Psi^j(t-t^j_n,x-x^j_n)-(K\ast|w_n|^2)w_n \right).
\end{aligned}
\end{equation*}
First of all, one observes that $w_n(0)-u_n(0)=w_0-u_{0,n}=\rho^J_n$ and therefore by \eqref{small:rem:nonlin}
\begin{equation*}
\limsup_{J\to\infty}\left(\lim_{n\to\infty}\|U(t)(w_n(0)-u_{0,n})\|_{L^8L^4}\right)=0.
\end{equation*}
Moreover it can be claimed that $\|\int_0^tU(t-s)e_n(s)\,ds\|_{L^8L^4}\leq\varepsilon$  uniformly  in $n,$ for $n$ large enough depending on $\varepsilon$  and $J.$ For these details we refer the reader to \cite{HR}. All of these ingredients allow us to apply \autoref{lemma3} obtaining therefore that 
\begin{equation*}
\sup_{n\in\N}\|u_n\|_{L^8L^4}\leq C<\infty,
\end{equation*}
which is a contradiction with respect to \eqref{contra:hyp}.

We eventually arrive to the existence of only one nontrivial profile and again by proceeding in the same spirit of \cite{HR} we get the existence of a global non-scattering solution as in the statement of \autoref{lemcri}.
\end{proof}

\begin{prop}
The translation path $x(t):\R^+\mapsto\R^3$ has a sub-linear growth at infinity, namely as $t\to\infty$
\begin{equation}\label{sublin-decay}
\frac{|x(t)|}{t}=o(1).
\end{equation}
\end{prop}
\begin{proof}
The proof of this spatial control is contained in \cite{DHR}, once one notices that also for \eqref{GP} the momentum of the critical solution $u_{sl}$ given in \autoref{lemcri} is zero, i.e. $P(u_{sl})=\Im\{ \int_{\mathbb R^3} \bar u_{sl}\nabla u_{sl}\,dx\}=0.$ In fact, if we consider the Galilean transformation of $u_{sl}$ given by
\begin{equation*}
T_v(u_{sl}(x,t))=e^{i(v \cdot x -|v|^2t)}u_{sl}(x-2vt, t),\qquad v\in\R^3
\end{equation*}
and we assume that $P(u_{sl})\neq0,$ then by selecting 
\begin{equation*}
v=-\frac{P(u_{sl})}{\mathcal{M}(u_{sl})}
\end{equation*}
we will have $P(T_{v}(u_{sl}))=0$ and $\mathcal{E}(T_{v}(u_{sl}))< \mathcal{E}(u_{sl})<\gamma(c);$ moreover the function  
\begin{equation*}
\mu \mapsto \mathcal{E}(T_{\mu  v}(u_{sl}))
\end{equation*}
is decreasing for $0<\mu<1$. Let us suppose therefore that $\mathcal{G}(T_{ v}(u_{sl}))<0;$ by continuity there exists $0<\bar\mu<1$ such that 
\begin{equation*}\mathcal{G}(T_{\bar\mu v}(u_{sl}))=0, \quad \mathcal{E}(T_{\bar\mu v}(u_{sl}))< \mathcal{E}(u_{sl})<\gamma(c)
\end{equation*}
which is impossible, therefore $\mathcal{G}(T_{ v}(u_{sl}))>0$. But this implies that $T_{ v}(u_{sl})$ does scatter and therefore $u_{sl}$ cannot be the minimal element.
\end{proof}
Finally, we state the uniform localization  of  $\left\{ u_{sl}(t,x-x(t)),\ t\geq0\right\},$ which is a standard consequence of the precompacteness property of the minimal element. 
\begin{prop}\label{prop:5.4} For any $\varepsilon>0$ there exists a radius $\rho=\rho(\varepsilon)>0$ such that 
\begin{equation*}
\begin{aligned}
\int_{|x+x(t)|>\rho}|u_{sl}(t)|^2+|\nabla u_{sl}(t)|^2+\left|\lambda_1|u_{sl}|^4+\lambda_2(K\ast|u_{sl}(t)|^2)|u_{sl}(t)|^2\right|\,dx\\
\leq C(\lambda_1,\lambda_2)\int_{|x+x(t)|>\rho}|u_{sl}(t)|^2+|\nabla u_{sl}(t)|^2+|u_{sl}(t)|^4\,dx<\varepsilon,\qquad \forall\,t\geq0.
\end{aligned}
\end{equation*}
\end{prop}

\section{Extinction of the Minimal Element}

This section is devoted to the conclusion of the proof of \autoref{main} with the so-called rigidity part of the Kenig and Merle road map. The minimal global non-scattering solution built in \autoref{lemcri} can be only the trivial one, obtaining therefore a contradiction  with respect to the fact that its $L^8L^4$-norm is not finite. It is based on a convexity argument on the localized variance of the minimal element.
For the infinite-variance NLS equation, this method of considering a localized version of the variance was pioneered by Ogawa and Tsutsumi, see \cite{OT}, in order to show finite time singularity formation as an extension of the result by Glassey in a framework with finite variance, see \cite{Gla}.
\subsection{Localized Virial Identities} To lighten the notation, since now on we  simply write $u$ instead of $u_{sl}.$
Define $z_R(t)=R^2\int \chi\left(\frac xR\right)|u(t,x)|^2\,dx$ where $\chi\in\mathcal C_c^\infty(\R^3)$ is a cut-off function. Standard computations yield  
\begin{equation}\label{virial1}
\begin{aligned}
\frac{d}{dt}z_R(t)&=2\Im\left\{R\int\nabla\chi\left(\frac xR\right)\cdot\nabla u\bar u\,dx\right\}
\end{aligned}
\end{equation}
and to the immediate estimate 
\begin{equation*}
\left|\frac{d}{dt}z_R(t)\right|\lesssim R\|u\|_{L^2}\|\nabla u\|_{L^2}.
\end{equation*}
By using \eqref{virial1} and the equation solved by $u,$ i.e. \eqref{GP}, we have 
\begin{equation*}
\begin{aligned}
\frac{d^2}{dt^2}z_R(t)&=2\int\left(\nabla^2\chi\left(\frac xR\right)\nabla u\right)\cdot\nabla\bar u\,dx\\
&-\frac{1}{2R^2}\int\Delta^2\chi\left(\frac xR\right)|u|^2\,dx\\
&+\lambda_1\int\Delta\chi\left(\frac xR\right)|u|^4\,dx\\
&-2\lambda_2R\int\nabla\chi\left(\frac xR\right)\cdot\nabla\left(K\ast|u|^2\right)|u|^2\,dx.
\end{aligned}
\end{equation*}

\noindent If we choose $\chi(x)=|x|^2$ on $|x|\leq1$ and $supp\,(\chi)\subset B(0,2),$ then by direct computations we get 
\begin{equation}\label{eq:directcom}
\begin{aligned}
\frac{d^2}{dt^2}z_R(t)=\mathcal A-2\lambda_2R\mathcal B
\end{aligned}
\end{equation}
where 
\begin{equation*}
\begin{aligned}
\mathcal A&=4\int_{|x|\leq R}|\nabla u|^2\,dx+2\int_{R\leq |x|\leq 2R}\left(\nabla^2\chi\left(\frac xR\right)\nabla u\right)\cdot\nabla\bar u\,dx\\
&-\frac{1}{2R^2}\int_{R\leq|x|\leq2R}\Delta^2\chi\left(\frac xR\right)|u|^2\,dx\\
&+6\lambda_1\int_{|x|\leq R}|u|^4\,dx+\lambda_1\int_{R\leq|x|\leq2R}\Delta\chi\left(\frac xR\right)|u|^4\,dx
\end{aligned}
\end{equation*}
\noindent and
\begin{equation}\label{term:B}
\mathcal B=\mathcal B(|u|^2):=\int\nabla\chi\left(\frac xR\right)\cdot\nabla\left(K\ast|u|^2\right)|u|^2\,dx.
\end{equation}
It is trivial to estimate $\mathcal A$ as
\begin{equation}\label{eq:stimbas}
\mathcal A\geq 4\int |\nabla u|^2\,dx+6\lambda_1\int |u|^4\,dx -\varepsilon_1(R)
\end{equation}
where \begin{equation*}
\varepsilon_1(R) = C\left(\int_{|x|\geq R}|\nabla u|^2+R^{-2}|u|^2+|u|^4\,dx\right).
\end{equation*}
\\

Now  we focus on the more delicate term $\mathcal B.$
\\

Let us consider the second term appearing in the localized virial identity above, i.e.  $\mathcal B$ defined in \eqref{term:B}. Before starting with the analysis we recall some preliminary tools introduced by Lu and Wu in \cite{LW}, where the authors study the Davey-Stewartson equation. The nonlocal nonlinearity in that case is given by a convolution with a kernel having symbol $\frac{\xi_1^2}{|\xi|^2}$ instead of  dipolar kernel.

Let $\mathcal R_j f=\mathcal F^{-1}\left(-i\frac{\xi_j}{|\xi|}\hat f\right)$ the Riesz transform of $f,$ defined via the zero-order symbol $-i\frac{\xi_j}{|\xi|}.$ It is well-known that it maps $L^p$ into itself for any $p\in(1,+\infty).$ One recognizes that the symbol $-\frac{\xi_j^2}{|\xi|^2}$ is the one defining $\mathcal R_j^2.$ By recalling the expression of the dipolar kernel $K$ in Fourier variables \eqref{kernel:fou}, we get
%in Fourier variable of the dipolar kernel $K,$ see \eqref{kernel:fou}, we get 
\begin{equation*}
\begin{aligned}
\mathcal F\left(K\ast f\right)&=(\hat K \hat f)=\frac{4\pi}{3}\left(\frac{2\xi_3^2}{|\xi|^2}-\frac{\xi_2^2}{|\xi|^2}-\frac{\xi_1^2}{|\xi|^2}\right) \hat f\\
&=-\frac{4\pi}{3}\left(2\mathcal F(\mathcal R_3^2f)-\mathcal F(\mathcal R_2^2f)-\mathcal F(\mathcal R_1^2f)\right),
\end{aligned}
\end{equation*}
and therefore 
\begin{equation*}
\left(K\ast f \right) = -\frac{8\pi}{3}\mathcal R_3^2f+\frac{4\pi}{3}\mathcal R_2^2f+\frac{4\pi}{3}\mathcal R_1^2f.
\end{equation*}

After these considerations, we report the point-wise estimates contained in \cite{LW}.
\begin{lemma}{\cite[Lemma 5.1]{LW}} Let $f$ be  a smooth function,
and $\eta$ a  compactly supported function on $B(0,1)$ with $|\eta|\leq1.$  Then 
\begin{equation}\label{eq:pointwise1}
\left|\eta\left(\frac xR\right)\mathcal R^2_j\left(\left(1-\eta\left(\frac{\cdot}{4R}\right)\right)f\right)(x)\right|\lesssim \eta\left(\frac xR\right) R^{-3}\|f\|_{L^1(B(0,4R)^c)}
\end{equation}
and 
\begin{equation}\label{eq:pointwise2}
\left| \left(1-\eta\left(\frac{x}{4R}\right)\right)\mathcal R^2_j\left(\eta\left(\frac \cdot R\right)f\right)(x)\right|\lesssim \left(1-\eta\left(\frac{x}{4R}\right)\right) R^{-3}\|f\|_{L^1(B(0,2R))}.
\end{equation}
\end{lemma}
The above lemma is used to show a suitable estimate for the term $\mathcal B,$ yielding  the following.
\begin{lemma}\label{lemma:6.2} The term $\mathcal B$  satisfies
\begin{equation}\label{eq:varr}
-2\lambda_2R\mathcal B(|u|^2)\geq 6\lambda_2\int(K\ast|u|^2)|u|^2\,dx-\varepsilon_2(R)
\end{equation}
where \begin{equation*}
\varepsilon_2(R)= C\left(\int_{|x|\geq r}|\nabla u|^2+R^{-2}|u|^2+|u|^4\,dx\right),\qquad r\sim R.
\end{equation*}
\end{lemma}
The above lemma in conjunction with \eqref{eq:directcom} and \eqref{eq:stimbas} gives the following estimate.
\begin{prop}\label{prop:6.3} The localized variance satisfies 
\begin{equation*}
\frac{d^2}{dt^2} z_R(t)\geq4\mathcal G(u)-\varepsilon_1(R)-\varepsilon_2(R).
\end{equation*}
\end{prop}
\begin{proof}[Proof of \autoref{lemma:6.2}]
We introduce the functions $v_1=1_{\{|x|\leq10R\}} u$ and $v_2=(1-1_{\{|x|\leq10R\}})u,$ therefore $u=v_1+v_2,$ and we observe that, due to the disjointness of their supports, $|v_1+v_2|^2=|v_1|^2+|v_2|^2.$ Therefore, by linearity, 
\begin{equation*}
\begin{aligned}
\mathcal B&=\int\nabla\chi\left(\frac xR\right)\cdot\nabla\left(K\ast|v_1+v_2|^2\right)|v_1+v_2|^2\,dx\\
&=\int\nabla\chi\left(\frac xR\right)\cdot\nabla\left(K\ast|v_1|^2\right)|v_1|^2\,dx+\int\nabla\chi\left(\frac xR\right)\cdot\nabla\left(K\ast|v_2|^2\right)|v_1|^2\,dx\\
&+\int\nabla\chi\left(\frac xR\right)\cdot\nabla\left(K\ast|v_1|^2\right)|v_2|^2\,dx+\int\nabla\chi\left(\frac xR\right)\cdot\nabla\left(K\ast|v_2|^2\right)|v_2|^2\,dx
\end{aligned}
\end{equation*}
and since the support of $\nabla\chi\left(\frac xR\right)$ is contained in $B(0,2R)$ while the one of $v_2$ in $B(0,10R)^c,$ which are disjoint sets, it follows that 
\begin{equation}\label{B-split}
\begin{aligned}
\mathcal B&=\int\nabla\chi\left(\frac xR\right)\cdot\nabla\left(K\ast|v_1|^2\right)|v_1|^2\,dx+\int\nabla\chi\left(\frac xR\right)\cdot\nabla\left(K\ast|v_2|^2\right)|v_1|^2\,dx\\
&=\mathcal B_0+\mathcal B_1
\end{aligned}
\end{equation}
with
\begin{equation*}
\begin{aligned}
\mathcal B_1&=\int\nabla\chi\left(\frac xR\right)\cdot\nabla\left(K\ast((1-1_{\{|x|\leq10R\}})|u|^2)\right)|v_1|^2\,dx\\
&=-R^{-1}\int\Delta\chi\left(\frac xR\right)\left(K\ast((1-1_{\{|x|\leq10R\}})|u|^2)\right)|v_1|^2\,dx\\
&-2\int_{B(0,2R)}\left(K\ast((1-1_{\{|x|\leq10R\}})|u|^2)\right)\nabla\chi\left(\frac xR\right)\cdot\Re\{\bar v_1\nabla v_1\}\,dx\\
&=\mathcal B_{1,1}+\mathcal B_{1,2}.
\end{aligned}
\end{equation*}

\noindent We observe that 

\begin{equation*}
\begin{aligned}
|\mathcal B_{1,1}|&\lesssim R^{-1}\|K\ast((1-1_{\{|x|\leq10R\}})|u|^2)\|_{L^2}\|v_1\|^2_{L^4}\\
&\lesssim R^{-1} \|u\|_{L^4(B(0,10R)^c)}^2\|u\|_{L^4}^2\\
&\lesssim R^{-1} \|u\|_{L^4(B(0,10R)^c)}^2\|u\|_{H^1}^2,
\end{aligned}
\end{equation*}
while, by using  \eqref{eq:pointwise1}, 
\begin{equation*}
\begin{aligned}
|\mathcal B_{1,2}|&\lesssim \|1_{B(0,2R)}K\ast((1-1_{\{|x|\leq10R\}})|u|^2)\|_{L^\infty}\|\bar v_1\nabla v_1\|_{L^1}\\
&\lesssim R^{-3} \|u\|_{L^2(B(0,10R)^c)}^2\|\nabla v_1\|_{L^2}\|v_1\|_{L^2}\\
&\lesssim R^{-3} \|u\|_{L^2(B(0,10R)^c)}^2\| u\|_{H^1}^2.
\end{aligned}
\end{equation*}
By gluing everything together we eventually obtain 
\begin{equation}\label{eq:stimaB1}
|\mathcal B_1|\lesssim R^{-1}\left(\|u\|_{L^4(B(0,10R)^c)}^2 +R^{-2} \|u\|_{L^2(B(0,10R)^c)}^2\right)\| u\|_{H^1}^2.
\end{equation}

Then it remains to properly estimate the first integral in the right-hand side of \eqref{B-split}, namely $\mathcal B_0.$ Following the strategy in \cite{LW} we introduce the function $\tilde \chi_R=R^2\chi\left(\frac xR\right)-|x|^2$ and it straightforwardly yields  the equality
\begin{equation}\label{eq:Rb0}
\begin{aligned}
R\mathcal B_0&=R\int\nabla\chi\left(\frac xR\right)\cdot\nabla\left(K\ast|v_1|^2\right)|v_1|^2\,dx\\
&=\int\nabla\tilde\chi_R\cdot\nabla\left(K\ast|v_1|^2\right)|v_1|^2\,dx+\int\nabla(|x|^2)\cdot\nabla\left(K\ast|v_1|^2\right)|v_1|^2\,dx=\mathcal B_{0,1}+\mathcal B_{0,2}.
\end{aligned}
\end{equation}

\noindent By localizing again, by setting $v_1=w_1+w_2$ with $w_2=1_{\{|x|\leq R/10\}}v_1$ and noticing that the supports of $w_1$ and $w_2$ are disjoints, alike the one of $\nabla\tilde\chi_R$ and $w_2$,  we can split $\mathcal B_{0,1}$ in two further terms \begin{equation*}
\mathcal B_{0,1}=\int\nabla\tilde\chi_R\cdot\nabla\left(K\ast|v_1|^2\right)|v_1|^2\,dx=\int\nabla\tilde\chi_R\cdot\nabla\left(K\ast|w_1|^2\right)|w_1|^2\,dx+\mathcal B_{0,1}^{\prime\prime}=\mathcal B_{0,1}^{\prime}+\mathcal B_{0,1}^{\prime\prime}
\end{equation*}
where 
\begin{equation*}
\begin{aligned}
\mathcal B_{0,1}^{\prime\prime}&=\int\nabla\tilde\chi_R\cdot\nabla\left(K\ast(1_{\{|x|\leq R/10\}}|v_1|^2)\right)|w_1|^2\,dx\\
&=-\int\Delta\tilde\chi_R\left(K\ast(1_{\{|x|\leq R/10\}}|v_1|^2)\right)|w_1|^2\,dx-\int_{2R\leq|x|\leq10R}\left(K\ast(1_{\{|x|\leq R/10\}}|v_1|^2)\right)\nabla\tilde\chi_R\cdot\nabla(|w_1|^2)\,dx.
\end{aligned}
\end{equation*}
Similarly to the term $\mathcal B_1,$
\begin{equation*}
\begin{aligned}
\left|\int\Delta\tilde\chi_R\left(K\ast(1_{\{|x|\leq R/10\}}|v_1|^2)\right)|w_1|^2\,dx \right|&\lesssim \|K\ast(1_{\{|x|\leq R/10\}}|v_1|^2)\|_{L^2}\|w_1\|^2_{L^4}\\
&\lesssim\|u\|_{H^1}^2\|w_1\|^2_{L^4}\lesssim \|u\|_{H^1}^2\|u\|^2_{L^4(B(0,R/10)^c)}
\end{aligned}
\end{equation*}
while, by means of \eqref{eq:pointwise2} 
\begin{equation*}
\begin{aligned}
&\left| \int_{2R\leq|x|\leq10R}\left(K\ast(1_{\{|x|\leq R/10\}}|v_1|^2)\right)\nabla\tilde\chi_R\cdot\nabla(|w_1|^2)\,dx \right|\\
&\qquad\qquad\qquad\qquad\lesssim R \|1_{2R\leq|x|\leq10R}K\ast(1_{\{|x|\leq R/10\}}|v_1|^2)\|_{L^\infty}\|\nabla w_1\|_{L^2}\|w_1\|_{L^2}\\
&\qquad\qquad\qquad\qquad\lesssim R^{-2}\|u\|^2_{L^2}\|u\|_{H^1}\|u\|_{L^2(B(0,R/10)^c)}.
\end{aligned}
\end{equation*}
By summing up the two terms we end up with 
\begin{equation}\label{eq:stimB01primprime}
\begin{aligned}
|\mathcal B_{0,1}^{\prime\prime}|\lesssim \|u\|_{H^1}\left(\|u\|_{H^1}\|u\|^2_{L^4(B(0,R/10)^c)}+R^{-2}\|u\|^2_{L^2}\|u\|_{L^2(B(0,R/10)^c)}\right).
\end{aligned}
\end{equation}

It is left to estimate the term $\mathcal B_{0,1}^{\prime}=\int\nabla\tilde\chi_R\cdot\nabla\left(K\ast|w_1|^2\right)|w_1|^2\,dx.$ By setting $g=|w_1|^2=|1_{\{|x|\leq 10R\}}(1-1_{\{|x|\leq R/10\}})u|^2$ and making use of the Parseval identity,
\begin{equation*}
\begin{aligned}
\int\nabla\tilde\chi_R\cdot\nabla\left(K\ast|w_1|^2\right)|w_1|^2\,dx&=i\int \widehat {\nabla\tilde\chi_R g}(\xi)\cdot\xi\hat K\bar {\hat g}\,d\xi\\
&=\frac{i4\pi}{3}\int \widehat {\nabla\tilde\chi_R g}(\xi)\cdot\xi\overline{\left(2\mathcal F(\mathcal R_3^2g)-\mathcal F(\mathcal R_2^2g)-\mathcal F(\mathcal R_1^2g)\right) }\,d\xi.
\end{aligned}
\end{equation*}
Consider the generic term $\int \widehat {\nabla\tilde\chi_R g}(\xi)\cdot\xi\overline{\mathcal F(\mathcal R_j^2g)}\,d\xi$ ; it is explicitly given, up to (complex) constants, by
\begin{equation}\label{term:generic}
\begin{aligned}
\int \widehat {\nabla\tilde\chi_R g}(\xi)\cdot\xi \frac{\xi_j^2}{|\xi|^2}\bar{\hat g}(\xi)\,d\xi&=\int (\widehat{\nabla\tilde\chi_R}\ast\hat g)(\xi)\cdot\xi \frac{\xi_j^2}{|\xi|^2}\bar{\hat g}(\xi)\,d\xi\\
&=\iint \hat g(\eta)\widehat{\nabla\tilde\chi_R}(\eta-\xi)\cdot\left(\frac{\xi_1\xi}{|\xi|}\pm\frac{\eta_1\eta}{|\eta|}\right)\frac{\xi_1}{|\xi|}\hat g(\xi)\,d\eta\,d\xi\\
&=\int \nabla\tilde\chi_R\cdot\mathcal R_j(\nabla g)(x)\mathcal R_j\bar g(x)\,dx\\
&+\iint \hat g(\eta)\widehat{\nabla\tilde\chi_R}(\eta-\xi)\cdot\left(\frac{\xi_1\xi}{|\xi|}-\frac{\eta_1\eta}{|\eta|}\right)\frac{\xi_1}{|\xi|}\hat g(\xi)\,d\eta\,d\xi\\
&=-\frac12\int\Delta\tilde\chi_R|\mathcal R_j g(x)|^2\,dx+\iint \hat g(\eta)\widehat{\nabla\tilde\chi_R}(\eta-\xi)\cdot\left(\frac{\xi_1\xi}{|\xi|}-\frac{\eta_1\eta}{|\eta|}\right)\frac{\xi_1}{|\xi|}\hat g(\xi)\,d\eta\,d\xi
\end{aligned}
\end{equation}
since derivatives and Riesz transform commute. The first term in the right-hand side of \eqref{term:generic} is simply estimated by $\|u\|^4_{L^4(B(0,R)^c)}$ due to the continuity property of the Riesz transform. For the second term  we have 
\begin{equation*}
\begin{aligned}
\left|\iint \hat g(\eta)\widehat{\nabla\tilde\chi_R}(\eta-\xi)\cdot\left(\frac{\xi_1\xi}{|\xi|}-\frac{\eta_1\eta}{|\eta|}\right)\frac{\xi_1}{|\xi|}\hat g(\xi)\,d\eta\,d\xi\right|&\leq \left|\iint \hat g(\eta)\widehat{\nabla \chi_R}(\eta-\xi)\cdot\left(\frac{\xi_1\xi}{|\xi|}-\frac{\eta_1\eta}{|\eta|}\right)\frac{\xi_1}{|\xi|}\hat g(\xi)\,d\eta\,d\xi\right|\\
&+\left|\iint \hat g(\eta)\widehat{\nabla N}(\eta-\xi)\cdot\left(\frac{\xi_1\xi}{|\xi|}-\frac{\eta_1\eta}{|\eta|}\right)\frac{\xi_1}{|\xi|}\hat g(\xi)\,d\eta\,d\xi\right|
\end{aligned}
\end{equation*}
where $N=|x|^2$.
Now, since $\left|\frac{\eta\eta_j}{|\eta|}-\frac{\xi\xi_j}{|\xi|}\right|\lesssim |\eta-\xi|,$
\begin{equation*}
\begin{aligned}
\left|\iint \hat g(\eta)\widehat{\nabla\chi_R}(\eta-\xi)\cdot\left(\frac{\xi_1\xi}{|\xi|}-\frac{\eta_1\eta}{|\eta|}\right)\frac{\xi_1}{|\xi|}\hat g(\xi)\,d\eta\,d\xi\right|&\lesssim\iint|\hat g(\xi)\|\hat g(\eta)\|\eta-\xi|^2|\widehat{\chi_R}(\eta-\xi)|\,d\eta\,d\xi\\
&=\int|\hat g(\xi)|\int |\hat g(\eta)|\left\|\eta-\xi|^2\widehat{\chi_R}(\eta-\xi)\right|\,d\eta\,d\xi\\
&=\int|\hat g(\xi)|\left(|\hat g|\ast |\mathcal F(-\Delta \chi_R)|\right)\,d\xi\\
&=\int|\hat g(\xi)|\left( |\hat g|\ast \left| \mathcal F\left( -\Delta\chi\left(\frac \cdot R\right)\right)\right|\right)\,d\xi\\
&=\int|\hat g(\xi)|\left( |\hat g|\ast \left| \mathcal F\left( h\left(\frac \cdot R\right)\right)\right|\right)\,d\xi
\end{aligned}
\end{equation*}
where we defined $h(\cdot)=-\Delta\chi(\cdot),$ 
and continue in this way
\begin{equation*}
\begin{aligned}
\int|\hat g(\xi)|\left( |\hat g|\ast \left| \mathcal F\left( h\left(\frac \cdot R\right)\right)\right|\right)\,d\xi&\leq \int|\hat g(\xi)|\left( |\hat g|\ast \left| \mathcal F\left( h\left(\frac \cdot R\right)\right)\right|\right)\,d\xi\\
&\lesssim\|g\|_{L^2}\left\| |\hat g|\ast \left| \mathcal F\left( h\left(\frac \cdot R\right)\right)\right|\right\|_{L^2}\\
&\lesssim\|g\|_{L^2}^2\left\|  \mathcal F\left( h\left(\frac \cdot R\right)\right)\right\|_{L^1}\\
&\lesssim R^3 \|g\|_{L^2}^2\left\| \hat h(R\cdot) \right\|_{L^1}\\
&=\|g\|_{L^2}^2\left\| \hat h \right\|_{L^1}\\
&\lesssim \|g\|_{L^2}^2=\|w_1\|^2_{L^4}\leq\|u\|^2_{L^4(B(0,R/10)^c)},
\end{aligned}
\end{equation*}
since $\hat h\in L^1$ (being $\chi$ in the Schwartz class, hence $\Delta\chi$ is in the Schwartz class, so it is integrable).\\
In the same way
\begin{equation*}
\begin{aligned}
\left|\iint \hat g(\eta)\widehat{\nabla N}(\eta-\xi)\cdot\left(\frac{\xi_1\xi}{|\xi|}-\frac{\eta_1\eta}{|\eta|}\right)\frac{\xi_1}{|\xi|}\hat g(\xi)\,d\eta\,d\xi\right|&\lesssim\iint|\hat g(\xi)\|\hat g(\eta)\|\eta-\xi|^2|\widehat{N}(\eta-\xi)|\,d\eta\,d\xi\\
&=\int|\hat g(\xi)|\int |\hat g(\eta)|\left||\eta-\xi|^2\widehat{N}(\eta-\xi)\right|\,d\eta\,d\xi\\
&=\int|\hat g(\xi)|\left(|\hat g|\ast 6 \delta_0\right)\,d\xi\\
&\lesssim \|g\|_{L^2}^2\leq\|u\|^2_{L^4(B(0,R/10)^c)}.
\end{aligned}
\end{equation*}
Therefore 
\begin{equation}\label{eq:stimB01prime}
\left|\mathcal B_{0,1}^\prime \right|\lesssim \|u\|^2_{L^4(B(0,R/10)^c)}.
\end{equation}

Now we finish with the estimate of $\mathcal B_{0,2}=\int\nabla(|x|^2)\cdot\nabla\left(K\ast|v_1|^2\right)|v_1|^2\,dx.$ We observe that a direct application of the Parseval identity gives
\begin{equation*}
\begin{aligned}
\mathcal B_{0,2}&=2\int x\cdot\nabla\left(K\ast|v_1|^2\right)|v_1|^2\,dx=-2\int\nabla\cdot\left(x|v_1|^2\right)\left(K\ast|v_1|^2\right)\,dx\\
&=-6\int\left(K\ast|v_1|^2\right)|v_1|^2\,dx-2\int x\cdot\nabla(|v_1|^2)\left(K\ast|v_1|^2\right)\,dx=-3\int\left(K\ast|v_1|^2\right)|v_1|^2\,dx
\end{aligned}
\end{equation*}
(above $\nabla\cdot$ stands for the divergence operator) then, by writing $|v_1|=|v_1\pm u|,$ it is with a direct computation to produce, using the  Cauchy-Schwarz inequality and the continuity property of the dipolar kernel,  
\begin{equation}\label{eq:directcom2}
\mathcal B_{0,2}\geq-3\int\left(K\ast|u|^2\right)|u|^2\,dx+\tilde\varepsilon(r)
\end{equation}
where 
\begin{equation*}
\tilde\varepsilon(r)\sim\|u\|^2_{L^4(B(0, r)^c)}, \qquad r\sim R.
\end{equation*}
Now summing up \eqref{eq:directcom}, \eqref{B-split}, \eqref{eq:stimaB1}, \eqref{eq:Rb0}, \eqref{eq:stimB01primprime}, \eqref{eq:stimB01prime} with \eqref{eq:directcom2} we get the desired results stated in \eqref{eq:varr} of \autoref{lemma:6.2} and in \autoref{prop:6.3}.
\end{proof}

\subsection{Death of the soliton-like solution} In this section we can close the Kenig and Merle scheme by showing, through a convexity argument, that the soliton-like solution built in \autoref{min:el} is the trivial one, clearly reaching a contradiction with respect to its infinite spacetime norm. We still keep the convention $u=u_{sl}.$\\

\noindent By gluing the estimate in \autoref{prop:6.3} with the bound in \autoref{cor:lb} we get
\begin{equation*}
\frac{d^2}{dt^2}z_R(t)\geq \alpha-\varepsilon_1(R)-\varepsilon_2(R).
\end{equation*}
and we can finally conclude if we are able to show that also $\varepsilon_{1,2}(R)\to0$ as $R\to\infty,$ uniformly in time. Since they have qualitatively the same form, let us control just $\varepsilon_1(R).$
At this  point we can exploit a strategy as in \cite{DHR}, which allows us to conclude. In fact, consider two times $0<\tau<\tau_1$ and the interval $I=[\tau, \tau_1]$ and a radius $R\geq\sup_{I}|x(t)|+\rho$ where $\rho$ is as in \autoref{prop:5.4}. Then $\{|x|>R\}\subset\{|x+x(t)|>\rho\}$ and so $\varepsilon_1(R)\to0$ as $R\to\infty,$ which in turn implies, with the choice of $\varepsilon_{1,2}=\alpha/4$
\begin{equation*}
\begin{aligned}
\frac{d^2}{dt^2}z_R(t)\geq \frac\alpha2>0
\end{aligned}
\end{equation*}
for $R$ sufficiently large. Integrating on $I,$ we get 
\begin{equation*}
R\gtrsim R\|u\|_{L^2}\|\nabla u\|_{L^2}\gtrsim\left|\frac{d}{dt}z_R(\tau_1)-\frac{d}{dt}z_R(\tau)\right|\geq \frac\alpha2(\tau_1-\tau)
\end{equation*}
and by choosing $R=\rho+\delta\tau_1$ we get 
\begin{equation*}
\rho+\delta\tau_1\geq \beta (\tau_1-\tau),
\end{equation*}
for some $\beta>0.$ 
\begin{remark}
Since \eqref{sublin-decay} holds, it is always possible, once  $\delta>0$ has been selected,  to find $\tau=\tau(\delta)$ such that $|x(t)|\leq\delta t$ for any $t\geq\tau.$
\end{remark}
\noindent Therefore by choosing $\delta=\beta/2$ we get 
\begin{equation*}
\frac\beta2\tau_1\leq\rho+\beta\tau=\rho+\beta\tau\left(\frac\beta2\right)
\end{equation*}
which is a contradiction since the right-hand side of the above inequality is a finite constant, while the left-hand side diverges as $\tau_1\to\infty.$
We have eventually proved the following.
\begin{prop}
Let $u_{sl}$ be the precompact solution to \eqref{GP} constructed in the previous section. Then $u_{sl}\equiv0.$ 
\end{prop}
This last Proposition closes the concentration/compactness and rigidity method, since the trivial solution cannot have  a divergent spacetime norm.

\appendix

\section{}
In this first Appendix, we recall the Strichartz estimates. Beside their use in our work, they are the basic tool to study nonlinear dispersive equations of Schr\"odinger-type (but not only them), whose proof can be found in the classical monographs \cite{C, LP}, and \cite{KT} for the endpoint case ($r=6$ below). We refer the reader to these already mentioned works for more accurate treatments on these kind of a priori estimates for the Schr\"odinger propagator. We just point out here that they are essentially  consequences of the so-called dispersive estimate
\begin{equation}\label{disp:est}
\|U(t)f\|_{L^\infty}\leq C|t|^{-3/2}\|f\|_{L^1}, \qquad \forall\, t\neq0,\quad \forall\,f\in L^1,
\end{equation}
which also holds for general dimensions, namely in the whole space $\R^d$ with $L^1-L^\infty$ decay rate given by $|t|^{-d/2}.$ More generally, conservation of the $L^2$-norm along the linear propagation (which for the nonlinearity in \eqref{GP} also holds true for the nonlinear flow, see \eqref{eq:mass}) together with \eqref{disp:est} implies, as a trivial application of the Riesz-Thorin interpolation theorem, that for any $p\in[2,\infty]$ the $L^p-L^{p^\prime}$ bound below is satisfied:
\begin{equation}\label{disp:est2}
\|U(t)f\|_{L^p}\leq C|t|^{-\frac32\left(\frac12-\frac1p\right)}\|f\|_{L^{p^\prime}}, \qquad \forall\, t\neq0,\quad \forall\,f\in L^{p^\prime}.
\end{equation}

Let us now state the Strichartz estimates. \\
%\begin{lemma}\label{lemma:strich}

\noindent Let $(q, r)$, $(\gamma, \rho)$ be two arbitrary $3D-$admissible pairs, namely they satisfy the algebraic condition \begin{equation*}
\frac2q=3\left(\frac12-\frac1r\right),\qquad 2\leq r\leq6.
\end{equation*} 
Then for any interval $I\ni t_0$, bounded or unbounded,  
\begin{equation*}
\begin{aligned}
\|U(t) f\|_{L^q_IL^r}&\leq C_1 \|f\|_{L^2}, \qquad\qquad \forall\, f=f(x)\in L^2,\\
\left\|\int_{t_0}^{t}U(t-s) F(s)\,\right\|_{L^q_IL^r}&\leq C_2\|F\|_{L^{\gamma^\prime}_IL^{\rho^\prime}}, \qquad \forall\, F=F(t,x)\in L^{\gamma^\prime}_IL^{\rho^\prime}.
\end{aligned}
\end{equation*}
where the constant $C_1, C_2$  depend only on the structural parameter and not on the functions $f, F$ themselves.
%\end{lemma}
 
\begin{remark}
Due to the commutativity property between derivatives and the linear flow, the previous estimates extend to Sobolev spaces: 
\begin{equation*}
\begin{aligned}
\|U(t) f\|_{L^q_IW^{1,r}}&\leq \tilde C_1 \|f\|_{H^1}, \qquad\quad\quad\,\forall\, f=f(x)\in H^1, \\
\bigg\|\int_{t_0}^{t}U(t-s) F(s)\,ds\bigg\|_{L^q_IW^{1,r}}&\leq \tilde C_2\|F\|_{L^{\gamma^\prime}_IW^{1,\rho^\prime}},\qquad \forall\, F=F(t,x)\in L^{\gamma^\prime}_IW^{1,\rho^\prime}.
\end{aligned}
\end{equation*}
\end{remark}
\noindent We will also use an extension for non-admissible pairs for the inhomogeneous Strichartz estimates, see \cite{F} and \cite{Vil} for a general treatment.
%\begin{lemma}\label{foschi}
For any interval $I\ni t_0$, bounded or unbounded,  
\begin{equation}\label{foschi}
\bigg\|\int_{t_0}^{t}U(t-s) F(s)\,ds\bigg\|_{L^8_IL^4}\leq C_3\|F\|_{ L^{8/3}_IL^{4/3}},\qquad \forall\, F=F(t,x)\in L^{8/3}_IL^{4/3}.
\end{equation}
%\end{lemma}

\section{}\label{app}
In this appendix we prove \autoref{lemma3.1}, \autoref{lemma3.2}, \autoref{lemma3.3} and \autoref{lemma3}.

\begin{proof}[Proof of \autoref{lemma3.1}]
The proof in the stable regime can be shown as consequence of the coercivity of the energy, see \cite{CMS}. In the unstable regime, the proof is contained in \cite{BJ}. We sketch it. First of all, it is worth mentioning that under condition \eqref{UR} the energy could be negative, then a classical Glassey's argument would yield  finite time blowing-up solutions. Fix now $\lambda_2>0,$ $\lambda_1-\frac{4\pi}{3}\lambda_2<0.$ Thanks to \autoref{thm:1} it is sufficient to show that for initial data $u_0$ small enough in the $H^1$ space, then $\mathcal G(u_0)>0$ and $E(u_0)<\gamma(u_0).$ Let us recall that the potential energy can be written as \begin{equation*}
\mathcal P(u)=\left(\lambda_1-\frac{4\pi}{3}\lambda_2\right)\|u\|_{L^4}^4+\frac{4\lambda_2\pi}{(2\pi)^3}\int \frac{\xi_3^2}{|\xi|^2}\left(\widehat{|u|^2}\right)^2\,d\xi\geq\left(\lambda_1-\frac{4\pi}{3}\lambda_2\right)\|u\|_{L^4}^4.
\end{equation*}
Therefore  by using in order the Plancherel identity and the Sobolev embedding and moreover recalling we are working on $\lambda_1-\frac{4\pi}{3}\lambda_2<0$ we have 
\begin{equation*}
\begin{aligned}
\mathcal G(u)&\geq \mathcal T(u)+\frac{3}{2}\left(\lambda_1-\frac{4\pi}{3}\lambda_2\right)\|u\|_{L^4}^4\\
&\geq \|\nabla u\|_{L^2}^2-C \|u\|_{L^2}\|\nabla u\|_{L^2}^{3}>0
\end{aligned}
\end{equation*}
provided $\|u\|_{H^1}$ is small enough. 
\end{proof}

\begin{proof}[Proof of \autoref{lemma3.2}]
The proof is contained in \cite{BJ}, where it is shown that if the initial datum is small enough (and so the solution is global), this yields  uniform bound on the Strichartz norm $L^{8/3}W^{1,4}$ and this in turn implies that the solution scatters (see the monographs  \cite{C, LP}). The Duhamel's formulation of \eqref{GP} is 
\begin{equation*}
u(t,x)=U(t)u_0+i\int_{0}^tU(t-s)\left(\lambda_1|u|^2u+\lambda_2(K\ast|u|^2)u\right)(s)\,ds
\end{equation*} 
and by using the Strichartz estimates with $(q,r)=(\gamma,\rho)=(8/3,4)$ then $(q^\prime, r^\prime)=(8/5,4/3),$ by using the  H\"older inequality and the continuity property of the dipolar kernel, it is easy to get 
\begin{equation*}
\|u\|_{L^{8/3}W^{1,4}}\leq C\|u_0\|_{H^1}+C\|u\|^{5/3}_{L^{8/3}W^{1,4}}.
\end{equation*}
Let $\delta=\|u_0\|_{H^1};$ noticing that the set $S:=\{s \ s.t. \ f(s):=s-C\delta-Cs^{5/3}\leq0\}$ decomposes in two disjoint connected components, the continuity of the flow implies that if $\delta$ is sufficiently small, the $L^{8/3}W^{1,4}$-norm of $u$ is uniformly bounded for (positive) times. Scattering for (positive) times is an easy consequence of this uniform control and the definition of the scattering state. In fact, by defining $v(t)=U(-t)u$ and making use of the Duhamel's representation formula above, it is straightforward to check that 
\begin{equation*}
\|v(t_1)-v(t_2)\|_{H^1}\overset{t_{1,2}\to+\infty}\longrightarrow0.
\end{equation*}
Definition of $v$ and the unitary property of the linear propagator $U(t)$ eventually gives the result. The analysis for negative times is the same.
\end{proof}

\begin{proof}[Proof of \autoref{lemma3.3}]
If the solution $u$ to \eqref{GP} is global and such that $u(t,x)\in L^8L^4,$ then a perturbative argument shows that  $u(t,x)\in L^{8/3}W^{1,4},$ therefore concluding as in the proof of  \autoref{lemma3.2}. It is worth mentioning that in its generality this result was established by Cazenave and Weissler in their paper on the so-called \emph{rapidly decaying solutions}, see \cite{CW}. Since for any fixed $T$ the $u\in L^q_IW^{1,r}$ with $I=(0,T),$ let us consider $v(t)=u(t+T).$ It follows that 
\begin{equation*}
v(t,x)=U(t)u(T)+i\int_{0}^tU(t-s)\left(\lambda_1|v|^2v+\lambda_2(K\ast|v|^2)v\right)(s)\,ds
\end{equation*}
and by means of the Strichartz estimates, for $\tilde I=(0, \tilde T)$
\begin{equation*}
\begin{aligned}
\|v\|_{L^{8/3}_{\tilde I}W^{1,4}}&\leq C\|u(T)\|_{H^1}+C\|v\|
^2_{L^8_{\tilde I}L^4}\|v\|_{L^{8/3}_{\tilde I}W^{1,4}}\\
&\leq C+C\|u\|^2_{L^8_{(T,T+\tilde T)}L^4}\|v\|_{L^{8/3}_{\tilde I}W^{1,4}}
\end{aligned}
\end{equation*}
since $u$ is uniformly bounded in time in $H^1.$ It suffices to select $T\gg1$ such that $C\|u\|^2_{L^8_{(T,\infty)}L^4}\leq\frac12$ to obtain
\begin{equation*}
\sup_{\tilde T>0}\|u\|_{L^{8/3}_{\tilde I}W^{1,4}}<\infty \implies u\in L^{8/3}_{(0,\infty)}W^{1,4}.
\end{equation*}
For negative times the analysis is exactly the same.
\end{proof}

\begin{proof}[Proof of \autoref{lemma3}]

If the equation were reduced to the classical NLS equation \eqref{NLS}, then the proof would be contained in \cite{HR}. Since we are in the presence of the dipolar interaction term we will sketch the proof for sake of clarity.  Let $z=u-w;$ then $z$ satisfies 
\begin{equation}\label{diffe}
i\partial_{t}z+\frac12\Delta z=\lambda_1|u|^{2}u-\lambda_1|w|^{2}w+\lambda_2(K\ast|u|^2)u-\lambda_2(K\ast|w|^2)w-e
\end{equation} 
subject to initial condition $z_0=z(0,x)=u_0-w_0.$ Since $\|w\|_{L^8L^4}\leq M$ we can partition $[t_0, \infty)$ into $m = m(M)$ intervals $
I_j=[t_j , t_{j+1}]$ such that $\|w\|_{L^8_{I_j}L^4}\leq \delta$  for each $j,$ where $\delta$ is small enough (to be chosen later on). 
The integral formulation of  \eqref{diffe} is 
\begin{equation*}
z=U(t-t_j)z(t_j)+i\int_{t_j}^tU(t-s)\left(Z_1+Z_2\right)(s)\,ds
\end{equation*} 
where, as in \cite{HR} for NLS
\begin{equation*}
\begin{aligned}
Z_1&=|u|^2u-|w|^2w =|z+w|^2(z+w)-|w|^2w \\
&=w^2\bar z+2|w|^2z+2w|z|^2+\bar wz^2+|z|^2z+e
\end{aligned}
\end{equation*} 

\noindent while  the nonlocal nonlinearity splits as

\begin{equation*}
\begin{aligned}
Z_2&=(K\ast|u|^2)u-(K\ast|w|^2)w=(K\ast|z+w|^2)(z+w)-(K\ast|w|^2)w\\
&=(K\ast(|z+w|^2-|w|^2))w-(K\ast|z+w|^2)z,
\end{aligned}
\end{equation*} 
and  due to \eqref{foschi} on $I_j$
\begin{equation*}
\begin{aligned}
\left\|\int_{t_j}^{t} U(t-s)\left[\left((K\ast(|z+w|^2-|w|^2)\right)w\right](s)\,ds\right\|_{L^8_{I_j}L^4}&\lesssim\|(K\ast(|z+w|^2-|w|^2))w\|_{L^{8/3}_{I_j}L^{4/3}} \\
&\lesssim \||z+w|^2-|w|^2\|_{L^4_{I_j}L^2}\|w\|_{L^8_{I_j}L^4}\\
&\lesssim \|z\|^2_{L^{8}_{I_j}L^{4}}\|w\|_{L^8_{I_j}L^4}+\|z\bar w\|_{L^4_{I_j}L^2}\|w\|_{L^8_{I_j}L^4}\\
&\lesssim \|z\|^2_{L^{8}_{I_j}L^{4}}\|w\|_{L^8_{I_j}L^4}+\|z\|_{L^8_{I_j}L^4}\|w\|^2_{L^8_{I_j}L^4}
\end{aligned}
\end{equation*} 
and similarly 
\begin{equation*}
\begin{aligned}
\left\|\int_{t_j}^{t} U(t-s)\left[(K\ast|z+w|^2)z\right](s)\,ds\right\|_{L^8_{I_j}L^4}&\lesssim\|(K\ast|z+w|^2)z\|_{L^{8/3}_{I_j}L^{4/3}} \\
&\lesssim \||z+w|^2\|_{L^4_{I_j}L^2}\|z\|_{L^8_{I_j}L^4}\\
&\lesssim \left(\|z\|^2_{L^{8}_{I_j}L^{4}}+\| w\|^2_{L^{8}_{I_j}L^{4}}\right)\|z\|_{L^8_{I_j}L^4}\\
&\lesssim \|z\|^3_{L^{8}_{I_j}L^{4}}+\| w\|_{L^8_{I_j}L^4}^2\|z\|_{L^8_{I_j}L^4}
\end{aligned}
\end{equation*} 
hence, by using the hypothesis,
\begin{equation*}
\begin{aligned}
\left\|\int_{t_j}^{t} U(t-s)Z_2(s)\,ds\right\|_{L^8_{I_j}L^4}&\lesssim \delta\|z\|^2_{L^{8}_{I_j}L^{4}}+\delta^2\|z\|_{L^8_{I_j}L^4}+ \|z\|^3_{L^{8}_{I_j}L^{4}}
\end{aligned}
\end{equation*}
and therefore the nonlocal interaction term leads to the same estimate for  $Z_1$ contained in  \cite{HR}. By gluing up everything together we get 
\begin{equation*}
\|z\|_{L^{8}_{I_j}L^{4}}\leq \|U(t-t_j)z(t_j)\|_{L^{8}_{I_j}L^{4}}+c\delta\|z\|^2_{L^{8}_{I_j}L^{4}}+c\delta^2\|z\|_{L^8_{I_j}L^4}+c\|z\|^3_{L^{8}_{I_j}L^{4}}+c\varepsilon,
\end{equation*}
thus the proof can be concluded in the same way of \cite[Proposition 2.3]{HR}. We report here the strategy for sake of clarity. For $\delta$ small enough, 
\begin{equation}\label{A:4}
\|z\|_{L^{8}_{I_j}L^{4}}\leq 2\|U(t-t_j)z(t_j)\|_{L^{8}_{I_j}L^{4}}+2C\varepsilon
\end{equation}
and choosing $t=t_{j+1}$ in the integral representation of $z(t)$ one obtains
\begin{equation}\label{A:5}
U(t-t_{j+1})z(t_{j+1})=U(t-t_j)z(t_j)+i\int_{t_j}^{j_{j+1}}U(t-s)(Z_1+Z_2)(s)\,ds,
\end{equation}
and so analogously to the estimates above it follows that 
\begin{equation*}
\|U(t-t_{j+1})z(t_{j+1})\|_{L^{8}L^{4}}\leq \|U(t-t_j)z(t_j)\|_{L^{8}L^{4}}+C\delta^2\|z\|_{L^{8}_{I_j}L^{4}}+C\delta\|z\|^2_{L^{8}_{I_j}L^{4}}+C\|z\|^3_{L^{8}_{I_j}L^{4}}+C\varepsilon.
\end{equation*}
Summing up \eqref{A:4} and \eqref{A:5} we eventually obtain
\begin{equation*}
\|U(t-t_{j+1})z(t_{j+1})\|_{L^{8}L^{4}}\leq 2\|U(t-t_j)z(t_j)\|_{L^{8}L^{4}}+2C\varepsilon
\end{equation*}
and iterating on $j\in\N$ it can be concluded that 
\begin{equation*}
\|U(t-t_{j})z(t_{j})\|_{L^{8}L^{4}}\leq 2^j\|U(t-t_0)z(t_0)\|_{L^{8}L^{4}}+2(2^j-1)C\varepsilon\lesssim 2^{j+2}\varepsilon.
\end{equation*}
The smallness assumption on $\delta$ is now defined if $2^{m+2}\varepsilon$ is sufficiently small (depending on the absolute constants of the a priori estimates and of course depending on $m$ which in turn is depending on $M$ of the statement).
\end{proof}

\section*{Acknowledgements}
\noindent J. B. is partially  supported  by Project 2016 ``Dinamica di equazioni nonlineari dispersive'' of  FONDAZIONE DI SARDEGNA. The authors warmly thank the referee for the careful reading and for the suggestions given in order to improve a preliminary version of the paper.

\end{document}